\UseRawInputEncoding
\documentclass[11pt]{amsart}
\usepackage[UKenglish]{isodate}
\usepackage{amsmath}
\usepackage{amsthm}
\usepackage{amsbsy}
\usepackage{amssymb}
\usepackage{amsfonts}
\usepackage[dvips]{lscape}
\usepackage{xcolor}
\usepackage{amscd}
\usepackage[all,cmtip]{xy}
\usepackage{euscript}
\usepackage{caption}  
\usepackage{parskip}
\usepackage{enumerate}
\usepackage{yfonts}
\usepackage{xcolor}
\usepackage{graphicx}
\usepackage{fancyhdr}
\usepackage{tikz-cd}
\usepackage{colortbl}
\usepackage{comment}
\usepackage{fullpage}
\usepackage{verbatim}
\usepackage{subcaption}

\usetikzlibrary{positioning}
\usetikzlibrary{decorations.pathreplacing}

\usepackage[margin=1in]{geometry}
\usepackage[expansion=false]{microtype}

\theoremstyle{plain}

\theoremstyle{plain}
\newtheorem{theorem}{Theorem}[section]

\newtheorem{proposition}[theorem]{Proposition}
\newtheorem{lemma}[theorem]{Lemma}
\newtheorem{corollary}[theorem]{Corollary}

\newcommand{\res}{\text{Res}}
\newcommand{\dis}{\text{Disc}}
\theoremstyle{remark}

\theoremstyle{definition}
\newtheorem{definition}[theorem]{Definition}

\newcommand{\diag}{\textup{diag}}

\renewcommand{\Im}{\textup{Im}}
\renewcommand{\Re}{\textup{Re}}

\renewcommand{\vec}[1]{\boldsymbol{#1}}
\newcommand{\vol}{\textup{vol}}

\newcommand{\PP}{\mathcal{P}}
\newif\iffinalrun
\iffinalrun
\else
 \fi

\iffinalrun
  \newcommand{\need}[1]{}
  \newcommand{\mar}[1]{}
\else
  \newcommand{\need}[1]{{\tiny *** #1}}
  \newcommand{\mar}[1]{\marginpar{\raggedright\tiny Fix Me:  #1 }}\fi

\usepackage{amssymb,amsmath,amsfonts,amsthm,epsfig,amscd,stmaryrd}
\usepackage{stmaryrd}

\usetikzlibrary{arrows}

\pgfdeclarelayer{background}
\pgfsetlayers{background,main}

\pgfmathsetmacro{\myxlow}{-2}
\pgfmathsetmacro{\myxhigh}{2}
\pgfmathsetmacro{\myiterations}{6}

\title{Limiting Distributions of Conjugate Algebraic Integers}
\author{Bryce Joseph  Orloski and Naser Talebizadeh Sardari}

\address{Penn State department of Mathematics, McAllister Building, Pollock Rd, State College, PA 16802 USA}
\email{nzt5208@psu.edu,bjo5149@psu.edu}
\thanks{We would like to thank Professors Pagano, Rumely, Sarnak, Serre, and Smith for their comments on the earlier version of this work. The authors also thank Pennsylvania State University's Institute for Computational and Data Sciences for allowing us to use their ROAR servers for our numerical experiments.}

\begin{document}
\begin{abstract} 
A recent advance by Smith~\cite{Smith} establishes a quantitative 
converse (conjectured by Smyth and Serre) to Fekete's 
celebrated theorem for compact subsets of $\mathbb{R}$.
Answering a basic question raised by Smith, we formulate and prove 
a quantitative converse of Fekete for general symmetric compact
subsets of $\mathbb{C}$. We highlight and exploit the algorithmic nature of 
our approach to give concrete applications to abelian varieties over finite fields
and to extremal problems in algebraic number theory.
\end{abstract}
\maketitle
\section{Introduction}
\subsection{Background}
This paper is motivated by the following question. What compact sets $\Sigma$ in the complex plane $\mathbb{C}$ contain infinitely many sets of conjugate algebraic integers, and how are they distributed in $\Sigma$? The ideas of this project originated in the works of many, including Schur~\cite{Schur}, Fekete~\cite{Fekete1923}, Siegel~\cite{MR12092}, Robinson~\cite{MR175881},  Smyth~\cite{MR736460}, Rumely~\cite{MR3154724}, Serre~\cite{MR4093205}, and Smith~\cite{Smith} on algebraic integers. We refer the reader to the excellent books of Baker and Rumely~\cite{MR2599526}, McKee and Smyth~\cite{McKee} and Serre~\cite{serre_curves} on this subject.
We introduce the basic concepts  and state our theorems which answer some open problems raised in previous works.
\newline

Let $\Sigma$ be a compact subset of the complex plane. Let 
\begin{equation}\label{transfn}
d_{\Sigma}(n):= \max_{z_1,\dots,z_n\in \Sigma}\prod_{i<j}|z_i-z_j|^{\frac{2}{n(n-1)}}.
\end{equation}
Fekete proved the following limit exists and called it the transfinite diameter of $\Sigma$: 
\begin{equation}\label{transf}
d_{\Sigma}:=\lim_{n\to \infty} d_{\Sigma}(n).
\end{equation}
For example, the transfinite diameter of a circle of radius $r$ is $r$. The transfinite diameter of $\Sigma$ equals the \textit{capacity} of $\Sigma$~\cite[Chapter 2]{MR2730573}. 
Fekete~\cite{Fekete1923}, generalizing some results of Schur~\cite{Schur}, proved that if $d_\Sigma<1$, then there is only a finite number of irreducible monic integral polynomials such that all of their roots lie in $\Sigma$.  Very recently, Kollar and Sarnak~\cite{Sarnak}
applied this result of Fekete together with combinatorial arguments to give
general lower bounds for the size and shape of the spectrum of three regular graphs. Remarkably, these
bounds are sharp. 
\\

Note that Fekete's condition $d_{\Sigma}<1$ is optimal. For example, let $\Sigma$ be the unit circle. Then primitive roots of unity give an infinite number of irreducible integral polynomials such that their roots lie all in $\Sigma$. The strict converse is not true. For example, the circle of radius $r>1$ around the origin where $r$ is not an algebraic integer gives a counter-example. However, Fekete and  Szeg\"o \cite{MR72941} proved that if $\Sigma$ is symmetric about the real axis and $d_{\Sigma}\geq 1$, then any open set $D$ including $\Sigma$ will contain infinitely many sets of conjugate algebraic integers. Robinson~\cite{Robinson} proved an analogous theorem about real point sets using the properties of Chebyshev's polynomials of $\Sigma$. 
\\

Smyth showed that if $\mu$ is the limit of uniform probability measures attached to distinct conjugate algebraic integers 
all of whose conjugates lie in a compact set $\Sigma$ of the complex plane, then 
\[
 \int \log|Q(x)|d\mu(x)\geq 0
\]
for every $Q(x)\in \mathbb{Z}[x]$.
In a recent breakthrough, Smith~\cite{Smith} proved that these necessary conditions are sufficient for $\mu$ to be a limiting measure  when $\Sigma$ is 
a subset of $\mathbb{R}$ satisfying some conditions.  In this paper, we establish this result for essentially any compact symmetric subset of the complex plane $\Sigma$ and our new method 
gives a different treatment of Smith's Theorem. In particular, we introduce some new constraints on the limiting measures of conjugate algebraic integers.  For a limiting  measure $\mu$, we prove in Theorem~\ref{mdim} that
\[
\int \log(|Q(x_1,\dots,x_n)|)d\mu(x_1)\dots d\mu(x_n)\geq 0
\]
for every $Q(x_1,\dots,x_n)\in \mathbb{Z}[x_1,\dots,x_n]$. In an upcoming work, we formulate a convex optimization problem based on the above new constraints and improve the lower bound on the Schur-Siegel-Smyth trace problem to the new best result 1.7982. Moreover, Theorem~\ref{mdim} allows us to improve some results of Smith~\cite[Proposition 3.5]{Smith} stated in Proposition~\ref{lowerd}.  This improvement is essential  in counting algebraic integers with a prescribed distribution ~\cite{2023arXiv230410021T}. 
\\

 Smith's method is based on the earlier works of Hilbert~\cite{MR1554854} that used the geometry of numbers and Robinson~\cite{Robinson}.  As mentioned earlier, Robinson used the properties of Chebyshev's polynomials of $\Sigma\subset \mathbb{R}$. Smith used some asymptotics of Chebyshev's polynomials proved in~\cite{Barry} that are not known for compact subsets of the complex plane. Our main technical invention in this paper is the new proof of  Robinson's result and Smith's results without using the properties of Chebyshev's polynomials; see Proposition~\ref{chpolreal}. 
 Our method is based on the  probabilistic sampling of roots with respect to the equilibrium measure and then deforming the roots  with a greedy algorithm along a gradient vector; see subsection~\ref{chebsec} and Proposition~\ref{chpolreal}. 
\\

In addition to the generalization of Smith's result to the complex plane, we also prove and implement a polynomial time algorithm which given an oracle to a probability measure $\mu$ (satisfying some technical conditions) on rectangles in the complex plane and a positive integer $n$, it outputs a degree $n$ polynomial with root distribution modelling $\mu$. Our algorithm and its implementation have no prior analogue to the best of our knowledge. Our numerical results show some new features of the algebraic integers that we discuss further in subsection~\ref{complex}.

 \subsection{Main results}\label{result}

 \subsubsection{Arithmetic probability measures}

Suppose that $\Sigma\subset \mathbb{C}$ is compact and symmetric about the real axis with $d_{\Sigma}\geq 1.$ Let $\mathcal{P}_\Sigma$ be the space of probability measures supported on $\Sigma$ equipped with the weak* topology. Let  
\[
\Sigma(\rho):=\{z\in \mathbb{C}: |z-\sigma|<\rho \text{ for some }\sigma\in\Sigma \}.
\]
If $\Sigma\subset \mathbb{R}$, let
\[
\Sigma_{\mathbb{R}}(\rho):=\{x\in \mathbb{R}: |x-\sigma|<\rho \text{ for some }\sigma\in\Sigma \}.
\]
Note that $\Sigma_{\mathbb{R}}(\rho)\subset \mathbb{R}$ and is only defined for $\Sigma\subset \mathbb{R}.$
\begin{definition}
  A measure $\mu \in \mathcal{P}_\Sigma$ is called a \textit{(real) arithmetic probability measure} if  $\mu$ is the weak* limit of a sequence of distinct uniform probability measures on a complete set of conjugate (totally real) algebraic integers  lying eventually inside ($\Sigma_{\mathbb{R}}(\varepsilon)$ for every $\varepsilon > 0$) $\Sigma(\varepsilon)$ for every $\varepsilon > 0$.
The set of all arithmetic probability measures is denoted (by $\mathcal{A}_{\Sigma_{\mathbb{R}}}\subset \mathcal{P}_\Sigma$) by $\mathcal{A}_{\Sigma}\subset \mathcal{P}_\Sigma.$ 
\end{definition}

Next, we define a convex subset of $\mathcal{P}_\Sigma$ that includes $\mathcal{A}_{\Sigma}.$
Let $P(x)$ be a polynomial with complex coefficients of degree $n$. Define its associated root probability measure on the complex plane to be
\[
\mu_P:=\frac{1}{n} \sum_{i=1}^{n} \delta_{\alpha_i},
\]
where $\alpha_i$ for $1 \leq i\leq n$ are the roots of $P$ and $\delta_{\alpha}$ is the delta probability measure at $\alpha.$ Suppose that $\mu\in \mathcal{A}_{\Sigma}.$
Then $\mu_{P_n}\stackrel{\ast}{\rightharpoonup} \mu$ as $n\to\infty$ for some sequence $\{P_n \}$ of distinct irreducible monic polynomials with integral coefficients.  Since $P_n$ has real coefficients, its roots are symmetric about the real axis. Therefore, $\mu$ should also be symmetric about the real axis. Let $Q$ be a polynomial with integral coefficients. Note that if $P_n \nmid Q,$
\[
\int \log |Q(x)| d\mu_{P_n}(x)= \frac{\log|\res(Q,P_n)|}{\deg(P_n)} \geq 0,
\]
 where $\res(Q,P_n) \in \mathbb{Z}$ is the resultant of $P_n$ and $Q.$ By taking $n\to \infty,$ Serre~\cite{MR4093205} proved that 
 \begin{equation}\label{conds}
 \int \log |Q(x)| d\mu(x) \geq 0      
 \end{equation}
 for every non-zero $Q(x)\in \mathbb{Z}[x].$  Let $\mathcal{B}_{\Sigma}\subset \mathcal{P}_{\Sigma}$ be the set of all $\mu\in  \mathcal{P}_{\Sigma}$ that are symmetric about the real axis and satisfy~\eqref{conds} for every non-zero $Q(x)\in \mathbb{Z}[x].$ Since $d_{\Sigma}\geq 1,$ it follows that the equilibrium measure of $\Sigma$ belongs to $\mathcal{B}_{\Sigma}$. 
  Note that $\mathcal{B}_{\Sigma}$ is a convex subset of $\mathcal{P}_{\Sigma}$ and  $\mathcal{A}_{\Sigma}\subset \mathcal{B}_{\Sigma}.$ Serre~\cite{MR2428512} proved that there exists  $\Sigma \subset \mathbb{R}^+$ and $\mu \in \mathcal{B}_{\Sigma}$ such that 
 $\int_{\mathbb{R}} x d\mu<1.8984.$ Smith~\cite{Smith} proved that if $\Sigma\subset \mathbb{R}$ satisfies certain technical conditions (including Serre's construction) then $\mathcal{A}_{\Sigma_{\mathbb{R}}}= \mathcal{B}_{\Sigma}.$  This gave a definite answer to the Schur--Siegel--Smyth Trace Problem. Finding $\mu \in \mathcal{B}_{\Sigma}$ with the minimal  expected value $\int_{\mathbb{R}} x d\mu$ is a linear programming problem
 and finding its optimal solution is an open problem; see~\cite{Smith}. It is the analog of finding the magic function in the linear programming problem formulated by Cohn-Elkies for the sphere packing problem.

\begin{theorem}\label{general}
Suppose that $\Sigma\subset \mathbb{C}$ is compact and symmetric about the real axis. If $d_{\Sigma}<1$, then $\mathcal{A}_{\Sigma}= \mathcal{B}_{\Sigma}=\emptyset.$ Otherwise, $d_{\Sigma}\geq 1$ and 
\(\mathcal{A}_{\Sigma}= \mathcal{B}_{\Sigma}\neq \emptyset.\) Moreover, if $\Sigma\subset \mathbb{R}$ is compact then \(\mathcal{A}_{\Sigma}= \mathcal{B}_{\Sigma}=\mathcal{A}_{\Sigma_{\mathbb{R}}}.\)
\end{theorem}
Next, we introduce some notation and state a corollary of our result that generalizes the main theorem of Smith~\cite{Smith} to the complex plane. 
Let 
\[
[a,b]\times[c,d]:=\left\{z\in \mathbb{C}: \Re(z)\in [a,b] \text{ and } \Im(z)\in [c,d]\right\}\subset \mathbb{C}.
\]
We call $[a,b]\times[c,d]$ a rectangle, and a real  interval if $c=d=0$.  
A measure $\mu$ is a H\"older probability measure if there exists $\delta>0$ and $A>0$ such that
\[
\mu([a,b]\times[c,d]) \leq A\max(|b-a|, |c-d|)^{\delta}
\]
for every $a<b, c<d.$
Suppose that $\Sigma \subset \mathbb{R}$ is a finite union of closed intervals with $d_{\Sigma}>1$ and $\mu \in \mathcal{B}_{\Sigma}$ is H\"older. 
Smith's main theorem~\cite[Theorem 1.6]{Smith} is equivalent to~\cite[Proposition 2.4]{Smith}, the existence of a sequence of  conjugate algebraic integers lying inside $\Sigma$ and equidistributing to $\mu.$ As we discussed, let $C_r\subset \mathbb{C}$ be  the circle of radius $r>1$ around the origin where $r$ is not an algebraic integer, and $\mu_r\in B_{C_r}$ be the uniform measure on $C_r$. Then $\mu_r$ gives a counter-example~\cite{MR72941} to Smith's theorem over $\mathbb{C}$. Knowing this, Smith raised the question of extending his result to $\Sigma\subset \mathbb{C}.$ We state a corollary of our main theorem that recovers Smith's theorem and generalizes it to the complex plane. 
\begin{corollary}\label{gsmith}
    Suppose that $\Sigma \subset \mathbb{C}$ is compact and symmetric about the real axis and is a finite union of disjoint rectangles and real intervals with $d_{\Sigma}>1$. If $\mu \in \mathcal{B}_{\Sigma}$, then there exists a sequence of  conjugate algebraic integers lying inside $\Sigma$ and equidistributing to $\mu.$ 
\end{corollary}
Smith's method does not directly imply  Corollary~\ref{gsmith}. For example, his method requires the complement of $\Sigma\subset \mathbb{C}$ to be connected. This condition is trivially satisfied for compact $\Sigma\subset \mathbb{R},$ but fails for some $\Sigma\subset \mathbb{C}$.
We introduce new ideas to prove the above theorem. This is discussed in Section~\ref{method}. 
\\

Next, we introduce some new notation to state a quantitative version of the  Theorem~\ref{general} under some assumptions. Let \begin{equation}\label{logpot}
    U_{\mu}(z):= \int \log|z-w| d\mu(w),
\end{equation}
which is called the logarithmic potential of $\mu$. In other works, this definition of the potential is replaced with $\int -\log|z-w|d\mu(w)$.
\newline

  We cite the following definition from Smith~\cite[Definition 2.7]{Smith}.
  \begin{definition}
      Fix a H\"older measure $\mu$ with support contained in the compact subset $\Sigma\subset \mathbb{C}$. Given a complex polynomial $P$ of degree less than or equal to $n$, define the $n$-norm of $P$ with respect to $(\mu,\Sigma)$ by
      \[
      \|P\|_n:=\max_{z\in \mathbb{C}}\left(e^{-nU_{\mu}(z)}|P(z)|  \right).
      \]
It follows from~\cite[Theorem III.4.2]{logpotentials} that if  $\Sigma\subset \mathbb{C}$ is compact and  has empty interior with connected complement and for a sequence of monic polynomials $P_n$ of degree $n$ for $n\geq 1$
\[
\limsup_n \|P_n\|_n^{1/n} \leq 1, 
\]
then the roots of $P_n$ are equidistibuted to $\mu$.

  \end{definition}
\begin{theorem}\label{main1}
Suppose $\Sigma\subset \mathbb{C}$ is compact and  has empty interior with connected complement.  Suppose that $\mu \in \mathcal{B}_{\Sigma}$ is a H\"older measure with exponent $\delta$ and $D$ is any open set containing $\Sigma$. For every large enough integer $n\geq 1$, there exists a monic irreducible polynomial $h_n$ of degree n such that all roots of $h_n$ are inside $D$ and 
\[
\|h_n\|_n\leq e^{n^{1-\delta'}}
\]
for some $\delta'>0$ which only depends on the H\"older exponent of $\mu$. As a result, the roots of $h_n$ are equidistibuted to $\mu.$ Moreover, if $\Sigma\subset \mathbb{R}$ is compact then it is possible to take any open set $D\subset \mathbb{R}$  in the real point set topology containing $\Sigma$  and prove the same result. 
\end{theorem}

\subsubsection{Complexity of  approximating arithmetic probability measures}\label{complex}
The next goal is to study the complexity of finding a sequence of distinct irreducible monic polynomials $\{P_n\}$
 such that $\{ \mu_{P_n} \}$ converges in the weak$^\ast$ topology to a given  $\mu \in \mathcal{B}_{\Sigma}.$

\begin{theorem}\label{main2}
Let $\Sigma$ be as in Theorem \ref{main1}.
Suppose that $\mu \in \mathcal{B}_{\Sigma}$ is a H\"older probability measure and $D$ is any open set containing $\Sigma$. There is a polynomial time algorithm in $n$ that returns an integral, irreducible, monic polynomial $g_n$ of degree $n$ such that all roots of $P_n$ are inside $D$, and 
\[
\|g_n\|_n\leq e^{\frac{Cn\log(\log(n))^3}{\log(n)}}
\]
where $C > 0$ depends only on $\mu$ and is independent of $n$. As a result, the roots of $g_n$ are equidistibuted to $\mu.$ Moreover, if $\Sigma\subset \mathbb{R}$ is compact then it is possible to take any open set $D\subset \mathbb{R}$  in the real point set topology containing $\Sigma$  and prove the same result.  
\end{theorem}
We implemented a version of this algorithm for arithmetic probability measures $\mu\in \mathcal{B}_{\Sigma}$ constructed by Serre~\cite{MR2428512}. Explicitly, let $\Sigma=[a,b],$
\[
\mu=c\mu_{[a,b]}+(1-c)\nu_{[a,b]},
\]
 $a=0.1715,$ $ b=5.8255,$  $c=0.5004,$ $\mu_{[a,b]}$ is the equilibrium measure on $[a,b]$, and $\nu_{[a,b]}$ is the pushforward of the equilibrium measure 
on $[b^{-1},a^{-1}]$ under the map $z\to 1/z.$
 Figure~\ref{fig} shows the density of $\mu$, and Figure~\ref{fig2} shows a histogram of the roots of the output polynomial of our algorithm for $n=100$.

%
\begin{figure}
     \centering
     \begin{subfigure}[b]{0.4\textwidth}
         \centering
         \includegraphics[width=\textwidth]{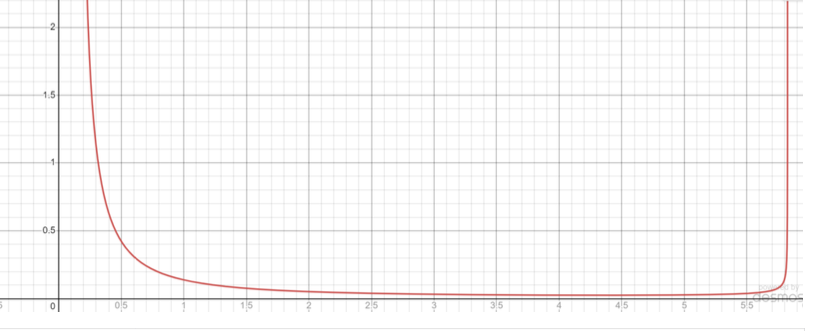}
         \caption{Density function of $\mu$}
         \label{fig}
     \end{subfigure}
     \hfill
     \begin{subfigure}[b]{0.4\textwidth}
         \centering         \includegraphics[width=\textwidth]{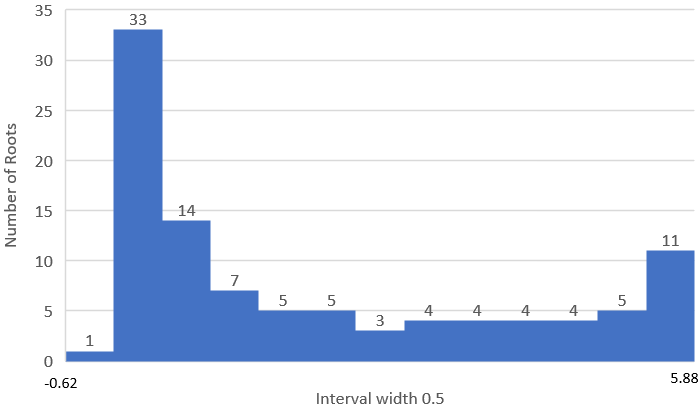}
                \caption{Number of roots of $P_{100}(x)$}
                \label{fig2}
     \end{subfigure}
\caption{This compares the density function of $\mu$ with the histogram of roots of our algorithm's output with input $n=100$.}\label{interval}
\end{figure}

\subsection{Applications to abelian varieties over finite fields} Recently, Shankar and  Tsimerman~\cite{shankar_tsimerman_2018}, \cite{Tsimerman1} conjectured that for every $g\geq 4$, every Abelian variety defined over $\overline{\mathbb{F}_q}$ is isogenous to the Jacobian of a curve defined over $\overline{\mathbb{F}_q}.$ By Honda Tate theory, the isogeny class of simple abelian varieties over $\mathbb{F}_q$ and isogenies over $\mathbb{F}_q$ corresponds  to algebraic integers such that all their conjugates have norm $\sqrt{q}.$ We apply our main theorem and Honda Tate theory to prove the following.
\begin{theorem}
    Given any finite field $\mathbb{F}_q$ and any integer $n$, our algorithm returns infinitely many isogeny classes of simple abelian varieties over $\mathbb{F}_q$ which are not isogenous to the Jacobian of any curve over $\mathbb{F}_{q^m}$ for every $1\leq m\leq n$.
\end{theorem}

Using the algorithm discussed in section \ref{main} of the paper, the authors found an explicit polynomial, with highest order terms being
$$x^{290} - 28x^{289} - {484}x^{288} + 20784x^{287} + \cdots,$$
which is irreducible with all roots inside $[-2\sqrt{3},2\sqrt{3}]\subset\mathbb R$ representing an isogeny class of an abelian variety defined over $\mathbb{F}_3$ that is not isogenous to the Jacobian of any curve over  $\mathbb F_9$.

\subsection{Method of proofs}\label{method}
In this section, we sketch  proofs of our main theorems.   Our main tool for studying the distribution of the roots of polynomials is 
Jensen's formula. Let $P(z)$ be a polynomial with complex coefficients and fix $w\in \mathbb{C}.$ If $z_1, \dots, z_m$ are the roots of $P(z)$ inside the disc $|z-w|<R$ and there is no zero on $|z-w|=R,$ then Jensen's formula states 
\begin{equation*}
    \frac{1}{2\pi} \int_{0}^{2\pi} \log |P(Re^{i\theta}+w)| d\theta - \log |P(w)| =\log \frac{R^m}{|z_1-w|\dots |z_m-w|}.
\end{equation*}
In particular, 
\begin{equation}\label{jensen}
    \frac{1}{2\pi} \int_{0}^{2\pi} \log |P(Re^{i\theta}+w)| d\theta - \log |P(w)| \geq 0. 
\end{equation}
Suppose that $P$ is a monic polynomial. We have $|P(w)|=|z_1-w|\cdots |z_{n}-w|$ where $n=\deg(P).$ Hence for a monic polynomial $P$, we have 
\[
        \frac{1}{2\pi} \int_{0}^{2\pi} \frac{\log |P(Re^{i\theta}+w)|}{n} d\theta  =\frac{m}{n}\log R+ \frac{1}{n}\sum_{i=m+1}^n\log|z_{i}-w|,
\]
where $|z_i-w|<R$ for $1\leq i\leq m$ and $|z_i-w|>R$ for $m+1\leq i\leq n.$ We rewrite the above as
\begin{equation}\label{Jensen}
\frac{1}{2\pi} \int_{0}^{2\pi} \frac{\log |P(Re^{i\theta}+w)|}{n} d\theta  = \int \log(z;w,R) d\mu_{P}(z),
\end{equation}
where
\[
\log(z;w,R):=\begin{cases} \log|z-w| &\text{ if } |z-w|>R, \\ \log{R} &\text{otherwise.}\end{cases}
\]
It follows from \eqref{Jensen}, \eqref{potapp}, and \eqref{potapp_real} that $\{ \mu_{P_n} \}$ converges weakly to a given probability measure $\mu$ if

\begin{equation}\label{conv1}
    \frac{\log|P_n(z)|}{n}= U_{\mu}(z)+o(1)
\end{equation}

for every $z\in \mathbb{C}$ such that $|z-z_{i,n}|>e^{-\sqrt{n}}$, where $z_{i,n}$ are complex roots of $P_n.$
We note that $\frac{\log|P_n(z)|}{n}$ is a harmonic function with singularities at roots of $P_n$ and 
\[\frac{\log|P_n(z)|}{n} = \log|z|+ O\left(\frac{1}{|z|}\right).\]
We note that $ U_{\mu}(z)$ has a similar behavior as $|z| \to \infty.$ Since $U_{\mu}(z)$ and $\frac{\log|P_n(z)|}{n}$ are both harmonic functions away from the roots of $P_n$ and the support of $\mu$ with similar asymptotic behavior at infinity, it follows from the mean value theorem for harmonic functions  that \eqref{conv1} is equivalent to 
\begin{equation}\label{cvx}
    \frac{\log|P_n(z)|}{n}\leq  U_{\mu}(z)+o(1)
\end{equation}
for every $z\in \mathbb{C}$ assuming $\mathbb{C}\setminus\Sigma$ is connected and the interior of $\Sigma$ is empty~\cite[Chapter III, Theorem 4.2]{logpotentials}.
Let 
\begin{equation}\label{defKn}
K_n:=\left\{ p(x)\in \mathbb{R}[x]: \deg(p)\leq n, \frac{\log |p(z)|}{n} \leq U_{\mu}(z) \text{ for every } z\in \mathbb{C}\right\}.
\end{equation}
$K_n$ is a symmetric, convex region inside the vector space of polynomials with real coefficients and degree less than or equal to $n$. For proving Theorem~\ref{main1}, we use Minkowski's second theorem to prove the existence of integral polynomials of degree $n$ inside $K_n$ up to a sub-exponential factor in $n$.  The key observation is that the necessary condition \eqref{conds} implies $K_n$ is well-rounded which means the successive minima of the lattice of integral polynomials with respect to $K_n$ are close to each other and a  lower bound on the volume of $K$ implies the existence of integral polynomials of degree $n$ inside $K_n$ up to a sub-exponential factor in $n$. The following theorem is an important technical result that allows us to use Minkowski's second theorem.
 \begin{theorem}\label{mdim}
 Suppose that $\mu$ is a probability measure with compact support on the complex plane. The following are equivalent:
 \begin{enumerate}
    \item \label{listcond1} $\mu$ is invariant by complex conjugation and~\eqref{conds} holds for every non-zero  $Q(x)\in \mathbb{Z}[x].$
    \item \label{listcond2} $\mu$ is invariant by complex conjugation and \begin{equation}\label{sconds}
 \int \log |Q_m(x_1,\dots,x_m)| d\mu(x_1)\dots d\mu(x_m) \geq 0      
 \end{equation}
  for every non-zero $Q_m(x_1,\dots,x_m)\in \mathbb{Z}[x_1,\dots,x_m].$  
 \end{enumerate}
 \end{theorem}

Recall that 
\[
\Sigma(\rho):=\{z\in \mathbb{C}: |z-\sigma|<\rho \text{ for some }\sigma\in\Sigma \}.
\] 
Note that for any open set $D$ containing $\Sigma,$ there exist $\rho>0$ such that $\Sigma(\rho)\subset D.$
\newline

For $\Sigma\subset \mathbb{C}$, we use  Rouch\'e's theorem to show that all roots of our irreducible polynomial are inside $\Sigma(\rho)$ and hence $D$. This is similar to the method of Fekete and Szeg\"o~\cite[Theorem I]{MR72941}.
For $\Sigma\subset \mathbb{R}$, Robinson used the Chebychev polynomials and constructed an integral polynomial of degree $n$ with $n$ sign changes on $\Sigma_{\mathbb R}(\rho)\subset \mathbb{R}$ to prove that the roots are real and inside $\Sigma_{\mathbb R}(\rho).$ Our results do not rely on the existence of Chebychev's polynomials on $\Sigma$; see subsection~\ref{chebsec} and Proposition~\ref{chpolreal}.
Proposition~\ref{chpolreal} allows us to improve some results of Smith~\cite[Proposition 4.1]{Smith}. 
\\

For proving Theorem~\ref{main2}, we apply a lattice algorithm due to Schnorr and find an irreducible, integral polynomial inside $e^{\frac{Cn\log(\log(n))^3}{\log(n)}}K_n$ in polynomial time in $n$ for some constant $C$ dependent on $\mu$. In fact, we take the polynomial to be monic and Eisenstein at the prime 2. Note that this differs from the strongest bound for which we prove existence which is $e^{Cn^{1-\delta_0}}K_n$ for some explicit $\delta_0>0$. Other versions of Schnorr's algorithm can achieve stronger bounds at a super-polynomial run time. 

\section{Minkowski's theorem applied to the lattice of integral polynomials}
\subsection{ Logarithmic potentials}
Recall the notation
in subsections~\ref{result} and~\ref{method}.  In this section, we assume that $\mu$ is supported on a compact set $\Sigma\subset \mathbb{C}$  and is a H\"older probability measure on the complex plane with exponent $\delta>0.$ Hence for every $a<b, c<d$, and some $A>0$,
\[
\mu([a,b]\times[c,d]) \leq A\max(|b-a|, |c-d|)^{\delta}.
\]
For two compact subsets $B, B'\subset \mathbb{C}$, let
 \[d(B,B'):=\inf_{\substack{z\in B\\ z'\in B'}} |z-z'|.\]
Recall the definition of the logarithmic potential of $\mu$ from~\eqref{logpot}:
$$
    U_{\mu}(z)= \int \log|z-w| d\mu(w).
$$
Finally, define the logarithmic energy of $\mu$ as
\[
I(\mu):=\int \log|z_1-z_2|d\mu(z_1)d\mu(z_2)=\int U_\mu(z)d\mu(z).
\]
We prove a lemma on the asymptotic behaviour of logarithmic potentials.
\begin{lemma}\label{lem1}
Suppose that $\mu$ is a probability measure with compact support inside $|z|<r$ for $z\in \mathbb{C}.$ If $|w|>2r$, we have
\[
U_{\mu}(w)=\log|w|+O\left(\frac{r}{|w|}\right).
\]
\end{lemma}
\begin{proof}
We have
\begin{align*}
    \left|U_{\mu}(w)-\log|w|\right|\leq  \int \left|\log |w-x| -\log|w| \right|d\mu(x)
\\
=\int \left|\log \left|1-\frac{x}{w}\right|\right| d\mu(x) \ll \frac{r}{|w|}.
\end{align*}
\end{proof}

\begin{lemma}\label{holderr}
Suppose that $\mu$ is a H\"older probability measure on the complex plane with exponent $\delta>0$. Then for any $z_1,z_2\in \mathbb{C}$ with $|z_1-z_2|<1/2,$ we have
\[
\left|U_{\mu}(z_1)-U_{\mu}(z_2)\right|\ll |z_1-z_2|^{\delta'},
\]
where $\delta'<\min(1/2,\delta/2)$ is any positive exponent.
\end{lemma}
\begin{proof}
    Suppose that $z_1,z_2\in \mathbb{C}$ and $|z_1-z_2|<1/2.$ We have

    \begin{align*}
        \left|U_{\mu}(z_1)-U_{\mu}(z_2)\right|&=\left| \int \log\frac{|z_1-w|}{|z_2-w|} d\mu(w)\right|.
    \end{align*}
We define the following covering of $\mathbb C\setminus\{z_1,z_2\}$:
\begin{align*}
U_0:&=\left\{w\in \mathbb{C}: |w-z_2|>2 \sqrt{|z_1-z_2|} \right\},
\\
U_1:&=\left\{w\in \mathbb{C}: |z_1-z_2|<\min(|w-z_2|,|w-z_1|)\leq 2\sqrt{|z_1-z_2|} \right\},
\\
U_k:&= \left\{w\in \mathbb{C}: |z_1-z_2|^k<\min(|w-z_2|,|w-z_1|)\leq |z_1-z_2|^{k-1}\right\},
\end{align*}
where  $k\geq 2$. Hence, 
\begin{align*}
        \left|U_{\mu}(z_1)-U_{\mu}(z_2)\right|\leq \sum_{i\geq 0}\left| \int_{U_i} \log\frac{|z_1-w|}{|z_2-w|} d\mu(w)
\right|.
    \end{align*}
For $w\in U_0$, we have $\left|\frac{z_1-z_2}{z_2-w}\right|<1/2$. We write the Taylor expansion of $\log$ and obtain 
\[
\left| \int_{U_0} \log\frac{|z_1-w|}{|z_2-w|} d\mu(w)
\right|=\left| \int_{U_0} \log \left| 1+\frac{z_1-z_2}{z_2-w}\right| d\mu(w)
\right| \leq \sum_{m\geq 1}\int_{U_0} \frac{1}{m}\left| \frac{z_1-z_2}{z_2-w}\right|^m d\mu(w)\ll |z_1-z_2|^{1/2}.
\]
For $w\in U_k$, where $k\geq 1$, we have 
\[
\left| \int_{U_k} \log\frac{|z_1-w|}{|z_2-w|} d\mu(w)
\right| \leq \mu(U_k)k\big|\log|z_1-z_2| \big|\ll |z_1-z_2|^{\frac{k\delta}{2}}k\big|\log|z_1-z_2|\big|,
\]
where we used $\mu(U_k)\ll |z_1-z_2|^{\frac{k\delta}{2}}$, because $\mu$ is a H\"older probability measure with exponent $\delta>0$.
 Therefore,
\[
 \left|U_{\mu}(z_1)-U_{\mu}(z_2)\right|\ll |z_1-z_2|^{1/2} +\sum_{k\geq 1}|z_1-z_2|^{k\delta/2}k\big|\log|z_1-z_2|\big|\ll |z_1-z_2|^{1/2}+|z_1-z_2|^{\delta/2}\big|\log|z_1-z_2|\big|.
\]
Finally, we have
\[
\left|U_{\mu}(z_1)-U_{\mu}(z_2)\right|\ll |z_1-z_2|^{\delta'},
\]
where $\delta'<\min(1/2,\delta/2)$ is any positive exponent.
\end{proof}


\subsection{Minkowski's theorem}
Recall the definition of $K_n$ from~\eqref{defKn}:
\[
K_n=\left\{ p(x)\in \mathbb{R}[x]: \deg(p)\leq n \text{ and } \frac{\log |p(z)|}{n} \leq U_{\mu}(z) \text{ for every } z\in \mathbb{C} \right\}.
\]
We identify the space of polynomials with real coefficients of degree less than or equal to $n$ with $\mathbb{R}^{n+1}$ by sending a polynomial to its coefficients. It is easy to see that $K_n$ is a symmetric convex region of $\mathbb{R}^{n+1}$.

\begin{lemma}\label{multlem}
      Suppose that $\mu$ is a probability measure with compact support inside $|z|<r$ for $z\in \mathbb{C}.$ Let $p(x)\in \lambda K_n$ for some $\lambda>0$ where $\deg(p)<n$. Then
\[
xp(x)\in r\lambda K_n.
\]  
\end{lemma}
\begin{proof} Let
\[
h(z):=U_{\mu}(z)-\frac{\log |zp(z)|}{n}.
\]
We note that $h(z)$ is a harmonic function on $|z|>r $ and outside roots of $p.$ By Lemma~\ref{lem1}, we have 
\[
 h(z)=\frac{n-\deg(p)-1}{n} \log|z|+\frac{\log|a_p|}{n}+O(1/|z|),
\]
where $a_p$ is the top coefficient of $p.$ By the above and since $\deg(p)< n$ and $h(z)$ is harmonic, the minimum of $h(z)$ is obtained inside  $ |z|<r$. Let $C:=\inf_{z\in \mathbb{C}} h(z)=\inf_{|z|<r} h(z).$ Since $p(x)\in \lambda K_n$, we have 
\[
|zp(z)| \leq r \lambda e^{n U_{\mu}(z)}
\]
for any $|z|<r.$ This implies that
\[
C\geq -\frac{\log r\lambda}{n}.
\]

Therefore,
\[
 e^{nU_{\mu}(z)-\log |zp(z)|}= e^{nh(z)}\geq e^{nC}\geq (r\lambda)^{-1}
\]
for any $z\in \mathbb{C}.$ This completes the proof of our lemma. 
\end{proof}
\begin{lemma}\label{derlem}
    Suppose that $p(x)\in \lambda K_n$ for some $\lambda>0.$ We claim that 
\[
p'(x)\in An^C\lambda K_n,
\]
where $p'$ is the derivative of $p$ and $A$ and $C$ are constants  that only depend on  $\mu.$
\end{lemma}
\begin{proof}
Let $D(z_0,r)$ be the disk centered at some $z_0\in \mathbb{C}$ and radius $0 < r < \frac{1}{2}.$ By Bernstein inequality~\cite[Corollary 5.1.6]{MR1367960}, for any $r>0$,
\[
\sup_{z\in D(z_0,r)} \left|p'(z)\right|  \leq \frac{\deg p}{r}\sup _{z\in D(z_0,r)} |p(z)|.
\]
Hence, we have
\[
|p'(z_0)|\leq \frac{n}{r}\sup _{z\in D(z_0,r)} |p(z)|  \leq \lambda \frac{n}{r}\sup _{z\in D(z_0,r)} e^{nU_{\mu}(z) }.
\]
By Lemma~\ref{holderr}, we have for any $z\in D(z_0,r)$,
\[
 e^{nU_{\mu}(z) } = e^{n(U_{\mu}(z)-U_{\mu}(z_0)) }e^{nU_{\mu}(z_0) }  \leq e^{A'nr^{\delta'}}e^{nU_{\mu}(z_0) }.
\]
for some $A'>0$ and $\delta'>0$ that only depends on $\mu.$ By taking $r=n^{-C'}$ for any $C'>1/\delta'$, it follows that 
\[
|p'(z_0)|\leq  \lambda \frac{n}{r} e^{A'nr^{\delta'}} e^{nU_{\mu}(z_0) } \leq A\lambda n^{C'+1}e^{nU_{\mu}(z_0) },
\]
for any $z_0\in \mathbb{C},$ where $A=e^{A'}$. This implies that $ p'(x)\in An^C\lambda K_n$ where $C=C'+1.$ 
\end{proof}

Let $\Gamma_n$ be the lattice of integral polynomials of degree less than or equal to $n$ embedded into $\mathbb{R}^{n+1}$. The successive minima of $K_n$ on $\Gamma_n$ are defined by setting the $k$-th successive minimum $\lambda_k$ to be the infimum of the numbers $\lambda$ such that $\lambda K_n$ contains $k+1$ linearly-independent vectors of $\Gamma_n$.
We have $0 < \lambda_0\leq \lambda_1 \leq \dots \leq \lambda_n <\infty$. Minkowski's second theorem states that 
\begin{equation}\label{minksec}
  \frac{2^{n+1}\vol(K_n)^{-1}}{(n+1)!} \leq  \lambda_0\lambda_1\dots \lambda_n\leq 2^{n+1}\vol(K_n)^{-1}.
\end{equation}



\begin{proposition}\label{successive}
 We have 
\[
\lambda_{m+1}\leq A n^{C}\lambda_{m}
\]
for any $0\leq m< n$, and some constants $A$ and $C$ that only depend on  $\mu.$
\end{proposition}
\begin{proof}

Suppose that $p_i\in \Gamma_n$ for $0\leq i\leq m$ form the set of linearly independent integral polynomials such that $p_i\in \lambda_i K_n.$ If for some $0\leq i\leq m$, $p_i'$ is linearly independent from $\{p_i:0\leq i\leq m \}$, then by Lemma~\ref{derlem},
\[
p'_i(x)\in An^C\lambda_m K_n.
\]
Hence,
\[
\lambda_{m+1}\leq A n^C\lambda_m
\]
which implies our proposition. Otherwise, $V_m:=\text{span}_{\mathbb{R}}\left<p_0,\dots,p_m\right>$ is closed under derivation. It follows that
\[
V_m=\text{span}_{\mathbb{R}}\left<1,x,\dots,x^m\right>.
\]
This implies that there exists $0\leq l\leq m$ such that  $\deg(p_j)\leq \deg (p_l)= m$ for every  $0\leq j\leq m.$ Let $q(x):=xp_l(x).$ Since $\deg q> m,$ $q$ is linearly independent of $V_m.$ By Lemma~\ref{multlem}, 
\[
q(x)\in A\lambda_m K_n
\]
for some constant $A>0$ that only depends on $\mu.$ Hence,
\[
\lambda_{m+1}\leq A\lambda_m
\]
and this completes the proof of our proposition. 
\end{proof}

\subsection{Lower bound on the volume of $K_n$}\label{bound_Kn}
In this subsection, we give a lower bound on the Euclidean volume of $K_n$ as a subset of $\mathbb{R}^{n+1}.$ Our method is to find $n+1$ linearly independent polynomials of degree at most $n$ inside $K_n$ and compute the volume of the simplex given by the convex hull of the origin and these $n+1$ points. Since $K_n$ is symmetric, there are $2^{n+1}$ simplexes with the same volume and disjoint interior inside $K_n$ by choosing different signs for the $n+1$ vertices. The total volume of these simplexes give a lower bound for the volume of $K_n.$
\subsubsection{Finding a simplex inside $K_n$} For simplicity in this section,  we assume that $n$ is odd and we find a simplex inside $K_n$ and estimate its volume. In the even case, fix one of the $z$ values defined below to be fixed real number. The asymptotics will be the same. Recall that $\mu$ is invariant by complex conjugation. 
Let $[a,b]\times[-c,c]\subset \mathbb{C}$ be a rectangle containing  $\Sigma.$ 
Fix $0<M\in \mathbb{Z}$ and $0<L \in \mathbb{R}$ where $n^\varepsilon<M \le n^{1/2-\varepsilon}$ and $L<n^{1-\varepsilon}M^{-2}$ for some fixed $\varepsilon>0.$
We partition each side of the rectangle into $2M$ equal length sub-intervals, and obtain a partition of the rectangle as 
\[
[a,b]\times[-c,c]= \bigcup_{0\leq i,j< 2M}B_{ij},
\]
where $B_{ij}:=[a_i,a_{i+1}]\times[c_{j},c_{j+1}]$ and $a_i:=a+\frac{i(b-a)}{2M}$ and $c_j:=-c+\frac{j(2c)}{2M}.$ Since $\mu$ is a H\"older measure with exponent $\delta$,  we have 
\begin{equation}\label{holder}
    \mu(B_{ij})\ll M^{-\delta}
\end{equation} 
for every $0\le i,j < 2M.$ Similarly, for $\Sigma\subset \mathbb{R}$, take $[a,b]$ containing $\Sigma$ and  write 
\[
[a,b]=\bigcup_{0\leq i< 2M}B_{i},
\]
where $B_{i}:=[a_i,a_{i+1}]$ and $a_i:=a+\frac{i(b-a)}{2M}$ where $n^\varepsilon < M < n^{1-\varepsilon}$ and $L<n^{1-\epsilon}M^{-1}$. 
Let 
\begin{equation}\label{n_ij}
    n_{ij}:=\lfloor (n+1)\mu(B_{ij})\rfloor+\varepsilon_{ij} \ll nM^{-\delta},
\end{equation}
where $\varepsilon_{ij} \in \mathbb{Z}$ and $|\varepsilon_{ij}| \leq L$  are chosen such that 
 $\sum_{i,j}n_{ij}=n+1$ and $n_{ij}=n_{i(2M-j-1)}.$
 Since we assumed $n$ is odd, this is possible.
By a covering argument, it is possible to find  $z_{ijk}\in B_{ij}$ for $0\leq i,j<2M$ and $1\leq k\leq n_{ij}$ such that $\overline{z_{ijk}}=z_{i(2M-j-1)k}$, and for every $w\in \partial B_{ij}$ and every  $1\leq k,k'\leq n_{ij}$,
\begin{equation}\label{Bij}
\begin{split}
 |z_{ijk}-w| &\gg n^{-A}
 \\
|z_{ijk}-z_{ijk'}|&\gg n^{-A}
\end{split}
\end{equation}
for some fixed $A>0.$
Similarly for $\Sigma\subset \mathbb{R}$, we define  $n_{i}:=\lfloor (n+1)\mu(B_{i})\rfloor+\varepsilon_{i} \ll nM^{-\delta}$ and $z_{ik}\in B_i$ for $1 \le k\leq n_i$ such that $|\varepsilon_{i}|<L,$  $\sum_{i}n_{i}=n+1$, and for every $w\in \partial B_{i}$ and every  $1\leq k,k'\leq n_{i}$,
\begin{equation}\label{Bireal}
\begin{split}
 |z_{ik}-w| &\gg n^{-A}
 \\
|z_{ik}-z_{ik'}|&\gg n^{-A} 
\end{split}
\end{equation}
for some fixed $A>0.$

\begin{proposition}\label{energy_approx}
We have 
\begin{equation}\label{diag}
\sum_{\substack{ij,i'j' \\ d(B_{ij},B_{i'j'})<M^{-1} }}\sum_{\substack{k,k' \\  i'j'k' \neq ijk}}\frac{\log|z_{ijk}-z_{i'j'k'}|}{(n+1)^2}=O(\log(n)M^{-\delta}),
\end{equation}
and
\begin{equation}\label{maineq}
\sum_{ijk}\sum_{ijk\neq i'j'k'} \frac{\log|z_{ijk}-z_{i'j'k'}|}{(n+1)^2} = 
I(\mu)+O\left(\frac{M^2\log(n)L}{n}+ M^{-\delta/2}\log(n)^{1/2}\right),
\end{equation}
where the implicit constant in the $O$ notation only depends on $\Sigma$, $\mu$, and $A$ as in \eqref{Bij}. 
Similarly, for $\Sigma\subset \mathbb{R}$, we have
\begin{equation}\label{diagreal}
\sum_{\substack{i,i' \\ d(B_{i},B_{i'})<M^{-1} }}\sum_{\substack{k,k' \\  i'k' \neq ik}}\frac{\log|z_{ik}-z_{i'k'}|}{(n+1)^2}=O(\log(n)M^{-\delta}),
\end{equation}
and
\begin{equation}\label{maineqreal}
\sum_{ik}\sum_{ik\neq i'k'} \frac{\log|z_{ik}-z_{i'k'}|}{(n+1)^2}=I(\mu)+O\left( \frac{M\log(n)L}{n}+ M^{-\delta/2}\right),
\end{equation}
where the implicit constant in the $O$ notation only dependents on $\Sigma$, $\mu$, and $A$ as in \eqref{Bireal}.
\end{proposition}
\begin{proof}
Note that for a pair $(i,j)$, there are at most $O(1)$ pairs $(i',j')$ where $d(B_{ij},B_{i'j'})<M^{-1}$.
By \eqref{n_ij} and \eqref{Bij}, we have 
\begin{equation*}
\sum_{\substack{ij,i'j' \\ d(B_{ij}, B_{i'j'})<M^{-1}}}\sum_{\substack{k,k' \\  i'j'k' \neq ijk}}\frac{\log|z_{ijk}-z_{i'j'k'}|}{(n+1)^2}\ll\sum_{\substack{ij,i'j' \\ d(B_{ij}, B_{i'j'})<M^{-1}}} \frac{n_{ij}n_{i'j'}\log(n)}{(n+1)^2}=O(\log(n)M^{-\delta}).
\end{equation*} 
This completes the proof of the first part of our proposition. We have
\begin{equation}\label{summ}
I(\mu)= \sum_{ij}\sum_{i'j'}\int_{z\in B_{ij}}\int_{z'\in B_{i'j'}} \log|z-z'|d\mu(z)d\mu(z').    
\end{equation}
Similarly, we have 
\[
\sum_{\substack{ij,i'j' \\ d(B_{ij},B_{i'j'})<M^{-1}}}\int_{z\in B_{ij}}\int_{z'\in B_{i'j'}} \log|z-z'|d\mu(z)d\mu(z')=O(\log(n)M^{-\delta}).
\]

Indeed, fix $z_0\in B_{ij}$ and let $B_{i'j'}^m:=B_{i'j'}\cap (B(z_0,M^{-1}2^{-m})\setminus B(z_0,M^{-1}2^{-m-1}))$. Then
\begin{align*}
\int_{B_{i'j'}}\log|z-z_0| & \ll \sum_{m=0}^\infty \int_{B_{i'j'}^m}\log|z-z_0|d\mu\\
&\ll \sum_{m=0}^\infty m\log(M)\left(2^{-m}M^{-1}\right)^\delta\\
&\ll_\delta \log(n)M^{-\delta}.
\end{align*}
Therefore,
$$\sum_{\substack{ij,i'j' \\ d(B_{ij},B_{i'j'})<M^{-1}}}\int_{z\in B_{ij}}\int_{z'\in B_{i'j'}} \log|z-z'|d\mu(z)d\mu(z')=O_\delta(\log(n)M^{-\delta}).$$

Suppose that $d(B_{ij},B_{i'j'})>M^{-1}.$ We have
\begin{multline}
\frac{1}{(n+1)^2}\sum_{k}\sum_{ k'} \log|z_{ijk}-z_{i'j'k'}|=\int_{z\in B_{ij}}\int_{z'\in B_{i'j'}} \log|z-z'|d\mu(z)d\mu(z')
\\
+O\left(\frac{M_1L}{n}\left(\mu(B_{ij})+\mu(B_{i'j'}) \right) +\frac{M_1L^2}{n^2} +\mu(B_{ij})\mu(B_{i'j'})M_2\right)
\end{multline}
where $M_1:=\left|\sup_{\substack{z\in B_{ij}\\ z'\in B_{i'j'}}} \log |z-z'|\right|$ and $M_2:=\sup_{\substack{z_1,z_2\in B_{ij}\\ z_1',z_2'\in B_{i'j'}}}\left|\log|z_1-z_1'|-\log|z_2-z_2'|\right|.$ We have 
\[
M_1\ll \log(n)
\]
and 
\[
M_2\ll  \frac{M^{-1}}{d(B_{ij},B_{i'j'})}.
\]
Therefore,
\begin{multline*}
\left|\sum_{ijk}\sum_{ijk\neq i'j'k'} \frac{\log|z_{ijk}-z_{i'j'k'}|}{(n+1)^2}-I(\mu) \right|  \ll \log(n)M^{-\delta}
\\+\sum_{\substack{ij,i'j' \\ d(B_{ij},B_{i'j'})>M^{-1}}} \frac{\log(n)L}{n}\left(\mu(B_{ij})+\mu(B_{i'j'}) \right)+
\frac{\log(n)L^2}{n^2}
+\mu(B_{ij})\mu(B_{i'j'})\frac{M^{-1}}{d(B_{ij},B_{i'j'})}.
\end{multline*}
Note that
\begin{equation}\label{square_sum}
\sum_{\substack{ij,i'j' \\ d(B_{ij},B_{i'j'})>M^{-1}}} \frac{\log(n)L}{n}\left(\mu(B_{ij})+\mu(B_{i'j'}) \right) \ll \frac{\log(n)M^2L}{n}.
\end{equation}
By the Cauchy-Schwarz inequality,
\[
\sum_{ij,i'j'}\mu(B_{ij})\mu(B_{i'j'})\frac{M^{-1}}{d(B_{ij},B_{i'j'})}\leq \left( \sum_{ij,i'j'} \mu(B_{ij})^2\mu(B_{i'j'})\right)^{1/2}\left(\sum_{ij,i'j'}\frac{M^{-2}\mu(B_{i'j'})}{d(B_{ij},B_{i'j'})^2} \right)^{1/2},\]
where the sum is over $ij,i'j'$ such that  $d(B_{ij},B_{i'j'})\geq M^{-1}.$ We have
\[
\sum_{ij,i'j'} \mu(B_{ij})^2\mu(B_{i'j'}) \ll M^{-\delta} \sum_{ij,i'j'} \mu(B_{ij})\mu(B_{i'j'})=M^{-\delta} .
\]
Moreover, 
\begin{equation}\label{recip_square_sum}
\sum_{ij,i'j'}\frac{M^{-2}\mu(B_{i'j'})}{d(B_{ij},B_{i'j'})^2} \leq \max_{i'j'} \sum_{ij}\frac{M^{-2}}{d(B_{ij},B_{i'j'})^2}\ll \log(M).
\end{equation}
Therefore,
\[
\left|\sum_{ijk}\sum_{ijk\neq i'j'k'} \frac{\log|z_{ijk}-z_{i'j'k'}|}{(n+1)^2}-I(\mu) \right|  \ll \log(n)M^{-\delta}+\frac{\log(n)M^2L}{n}+M^{-\delta/2}\log(M)^{1/2}.
\]
The real case is almost identical. To get the real bounds, we see that equations \eqref{square_sum} and \eqref{recip_square_sum} change. Equation \eqref{square_sum} changes to
\[
\sum_{\substack{i,i' \\ d(B_i,B_{i'})>M^{-1}}} \frac{\log(n)L}{n}\left(\mu(B_i)+\mu(B_{i'}) \right) \ll \frac{\log(n)ML}{n}.
\]
Lastly, \eqref{recip_square_sum} changes to
\[
\sum_{i,i'}\frac{M^{-2}\mu(B_{i'})}{d(B_{i},B_{i'})^2} \leq \max_{i'} \sum_{i}\frac{M^{-2}}{d(B_{i},B_{i'})^2}\ll 1.
\]
Therefore for $\Sigma\subset\mathbb R$, $\sum_{ik}\sum_{ik\neq i'k'}\frac{\log|z_{ik}-z_{i'k'}|}{(n+1)^2} = I(\mu)+O(M^{-\delta/2} + \frac{M\log(n)L}{n})$.
\end{proof}

Let $\mathcal{Z}:=\{z_{ijk}: 0\leq i,j<2M , 1\leq k\leq n_{ij} \}$ and 
\begin{equation}\label{measpol}
   p_{\mathcal{Z}}(x):=\prod_{i,j,k} (x-z_{ijk}).
\end{equation}
For $e,f< 2M$ and $g\leq n_{ef}$,
\begin{equation}\label{pmdef}
\begin{split}
p_{\mathcal{Z}}^{efg+}(x)&:=\frac{\frac{p_{\mathcal{Z}}(x)}{(x-z_{efg})}+\frac{p_{\mathcal{Z}}(x)}{(x-\overline{z_{efg}})}}{2},
        \\
p_{\mathcal{Z}}^{efg-}(x):&=\frac{\frac{p_{\mathcal{Z}}(x)}{(x-z_{efg})}-\frac{p_{\mathcal{Z}}(x)}{(x-\overline{z_{efg}})}}{2i\Im(z_{efg})}.
\end{split}
\end{equation}
Similarly for $\Sigma\subset \mathbb{R}$,
let $\mathcal{Z}:=\{z_{ik}:0\le i<2M, 1\le k\le n_i\}$ and
\begin{equation}\label{measpolreal}
   p_{\mathcal{Z}}(x):=\prod_{i,k} (x-z_{ik}), 
\end{equation}
and for $e< 2M$ and $g\leq n_{e}$,
\begin{equation}\label{pmdefreal}
p_{\mathcal{Z}}^{eg}(x):=\frac{p_{\mathcal{Z}}(x)}{(x-z_{eg})}.
\end{equation}
For each $z\in\mathbb C$, let $\hat z\in\mathcal Z$ minimize $|z-\hat z|$.

\begin{lemma}
We have $p_{\mathcal{Z}}(x)\in \mathbb{R}[x]$ and 
\[
p_{\mathcal{Z}}^{efg+}(x)=\frac{p_{\mathcal{Z}}(x)}{(x-z_{efg})(x-\overline{z_{efg}})}(x-\Re(z_{efg}))=p_{\mathcal{Z}}^{e(2M-f-1)g+}(x)
\]
and 
\[
p_{\mathcal{Z}}^{efg-}(x)=\frac{p_{\mathcal{Z}}(x)}{(x-z_{efg})(x-\overline{z_{efg}})}=p_{\mathcal{Z}}^{e(2M-f-1)g-}(x).
\]
In particular, $p_{\mathcal{Z}}^{efg\pm}(x)\in \mathbb{R}[x],$ $\deg(p_{\mathcal{Z}}^{efg+}(x))=n$, and $ \deg p_{\mathcal{Z}}^{efg-}(x)=n-1$.
\end{lemma}
\begin{proof}
It follows easily from $\overline{z_{ijk}}=z_{i(2M-j-1)k}.$
\end{proof}

\begin{proposition}\label{simplex} Suppose that $n \ge 2\max\{|a|,|b|,|c|\}.$
We have 
\[
p_{\mathcal{Z}}^{efg\pm}(x) \in \lambda^{efg\pm}K_n
\]
for some $\lambda^{efg\pm}>0$, where  
$$\frac{\log(\lambda^{efg\pm})}{n}\leq C\left(\frac{M^2L\log(n)}{n}+ M^{-\delta/2}\log(n)^{1/2}
\right)$$
for some constant $C$ that only depends on $\Sigma$, $\mu$, and $A$ from \eqref{Bij}. In particular, taking $M=\lfloor n^{1/3}\rfloor$ and $L=\lfloor n^{1/3-\delta/6}\rfloor,$ we obtain
\begin{equation}\label{simplex_potential}
|\lambda^{efg\pm}| \leq  e^{Cn^{1-\frac{\delta}{6}}\log(n)}.
\end{equation}
For $\Sigma \subset \mathbb R$, we get a stronger result. In particular, $p_{\mathcal{Z}}^{eg}(x) \in \lambda^{eg}K_n$ for some $\lambda_{eg}>0$ where
\[
\frac{\log(\lambda^{eg})}{n} \le C\left(\frac{ML\log(n)}{n}+M^{-\delta/2}\right).
\]
\end{proposition}
\begin{proof}
We give the proof for $p_{\mathcal{Z}}^{efg+}(x)$; the cases for $p_{\mathcal{Z}}^{efg-}$ and $p_{\mathcal{Z}}^{eg}$ follow from similar arguments. 
Note that $\deg(p_{\mathcal{Z}}^{efg+})=n$ and  
\[
\frac{\log |p_{\mathcal{Z}}^{efg+}(z)|}{n} = \int \log|z-x| d\mu_{p_{\mathcal{Z}}^{efg+}}= U_{\mu_{p_{\mathcal{Z}}^{efg+}}}(z).
\]
Hence if $|z|\geq n$ and $n \ge 2\max\{|a|,|b|,|c|\}$, then by
Lemma~\ref{lem1},
\[
\frac{\log |p_{\mathcal{Z}}^{efg+}(z)|}{n} - U_{\mu}(z)=O\left(\frac{1}{n}\right).
\]
So without loss of generality, we assume that $|z|<n$. For $z\in \mathbb{C},$ we have 
\begin{multline}\label{sum}
    \frac{\log |p_{\mathcal{Z}}^{efg+}(z)|}{n} - U_{\mu}(z)=\sum_{i,j}\left(\sum_{k\leq n_{ij}} \frac{\log |z-z_{ijk}|}{n} -\int_{B_{ij}} \log|z-x|d\mu(x)\right)
\\
+\frac{\log \frac{|z-\Re(z_{efg})|}{|z-z_{efg}||z-\overline{z_{efg}}|}}{n}.
\end{multline}
Let $d(z,B_{ij}):=\inf_{w\in B_{ij}} |z-w|.$ As noted previously, there are at most $O(1)$ pairs $(i,j)$ such that $d(z,B_{ij})<M^{-1}$. First, we give an upper bound on the sum in \eqref{sum} over $i,j$ where $d(z,B_{ij})\leq M^{-1}$.
Recall that $\hat z$ is the element of $\mathcal Z$ minimizing $|z-\hat z|$.
By \eqref{holder}, \eqref{n_ij} and \eqref{Bij}, we have 
\begin{equation}\label{sum1}
\sum_{i,j}\left(\sum_{k\leq n_{ij}} \frac{\log |z-z_{ijk}|}{n} -\int_{B_{ij}} \log|z-x|d\mu(x)\right)=\frac{\log |z-\hat{z}|}{n}+O(\log(n)M^{-\delta}),
\end{equation} 
where the sum is over $i,j$ where $d(z,B_{ij})\leq M^{-1}$.
This is because by \eqref{Bij}, $|z_{ijk}-z_{ijk'}|\gg n^{-A}$, and $d(\hat{z},\partial B_{ij})\gg n^{-A}$.
\newline

Next, we give an upper bound on the above sum over $i,j$, where $d(z,B_{ij})\geq M^{-1}$. Suppose that $d(z,B_{ij})\geq M^{-1},$ we have 
\[
\left|\sum_{k\leq n_{ij}} \frac{\log |z-z_{ijk}|}{n} -\int_{B_{ij}} \log|z-x|d\mu(x)\right|\leq \frac{M_1L}{n}+\mu(B_{ij})M_2.
\]
where $M_1:=\left|\sup_{z_1\in B_{ij}} \log |z-z_1|\right|$ and $M_2:=\sup_{z_1,z_2\in B_{ij}}\left|\log|z-z_2|-\log|z-z_1|\right|.$  Because we know $d(z,B_{ij})\geq M^{-1},$ we have
\[
M_1\ll \max (|\log(\text{diam}(\Sigma)+|z|)|, \log(M))\ll \log(n).
\]
 Suppose that $z_1,z_2\in B_{ij}.$ Let $r_1:=|z-z_1|$ and $r_2:=|z-z_2|$ and assume that $r_1  \leq r_2.$ By the mean value theorem, we have
\[
M_2=\frac{\left|r_1-r_2\right|}{r_3}
\]
for some $r_3,$ where $r_1 \leq r_3 \leq r_2.$ We have $|r_1-r_2|\leq |z_1-z_2|\leq M^{-1}$ and 
\[
r_3\gg d(z,B_{ij}).
\]

Hence,
\[
M_2\ll \frac{M^{-1}}{d(z,B_{ij})}.
\]
Therefore,
\begin{equation}\label{b1}
    \sum_{i,j}\left|\sum_{k\leq n_{ij}} \frac{\log |z-z_{ijk}|}{n} -\int_{B_{ij}} \log|z-x|d\mu(x) \right|\ll \frac{M^2L\log(n)}{n}+ \sum_{i,j}\mu(B_{ij})\frac{M^{-1}}{d(z,B_{ij})},
\end{equation}
where the sum is over $i,j,$ where $d(z,B_{ij})\geq M^{-1}$. By the Cauchy-Schwarz inequality,
\[
\sum_{i,j}\mu(B_{ij})\frac{M^{-1}}{d(z,B_{ij})}\leq \left( \sum_{i,j} \mu(B_{ij})^2\right)^{1/2}\left(\sum_{i,j}\frac{M^{-2}}{d(z,B_{ij})^2} \right)^{1/2}\ll M^{-\delta/2}\log(M)^{1/2},\]
where we used $\sum_{i,j}\mu(B_{ij})^2\ll M^{-\delta} \sum_{i,j}\mu(B_{ij})=M^{-\delta}$ and $\sum_{i,j}\frac{M^{-2}}{d(z,B_{ij})^2} \ll \log(M)$
where the sum is over $i,j,$ with $d(z,B_{ij})\geq M^{-1}.$ Therefore by \eqref{b1} and the above inequality, we have 
\begin{equation}\label{sum2}
    \sum_{i,j}\left|\sum_{k\leq n_{ij}} \frac{\log |z-z_{ijk}|}{n} -\int_{B_{ij}} \log|z-x|d\mu(x)\right| \ll 
    \frac{M^2L\log(n)}{n}+ M^{-\delta/2}\log(M)^{1/2},
\end{equation}
where the sum is over $i,j,$ where $d(z,B_{ij})\geq M^{-1}.$ Finally, by \eqref{sum}, \eqref{sum1} and \eqref{sum2}, we obtain
\begin{multline}\label{potapp}
    \frac{\log |p_{\mathcal{Z}}^{efg+}(z)|}{n} - U_{\mu}(z)= \frac{\log \frac{|z-\hat{z}||z-\Re(z_{efg})|}{|z-z_{efg}||z-\overline{z_{efg}}|}}{n}
    \\
    +O\left(\log(n)M^{-\delta}+ \frac{M^2L\log(n)}{n}+ M^{-\delta/2}\log(M)^{1/2}\right).
\end{multline}

By the definition of $\hat{z},$ $\frac{|z-\hat{z}||z-\Re(z_{efg})|}{|z-z_{efg}||z-\overline{z_{efg}}|}\leq 1$  and by letting $M=\lfloor n^{1/3}\rfloor$  and $L=\lfloor n^{1/3-\delta/6}\rfloor$, we get
\begin{equation*}
\frac{\log |p_{\mathcal{Z}}^{efg+}(z)|}{n} \leq  U_{\mu}(z)+O(n^{-\frac{\delta}{6}}\log(n)).
\end{equation*}

For the same reasoning as in Proposition \ref{energy_approx}, the proof for $\Sigma\subset\mathbb R$ is almost identical but
\begin{equation}\label{potapp_real}
    \frac{\log |p_{\mathcal{Z}}^{eg}(z)|}{n} - U_{\mu}(z)= \frac{\log \frac{|z-\hat z|}{|z-z_{eg}|}}{n}
    +O\left(\frac{ML\log(n)}{n}+ M^{-\delta/2}\right).
\end{equation}
\end{proof}

\subsubsection{Lower bound on the volume of the simplex}
Let 
\[
\mathcal M:=\left[p_{\mathcal{Z}}^{ijk\pm} \right],
\]
where $0\leq i<2M, 0\leq j< M, k\leq n_{ij}$
be the square matrix of size $(n+1)\times (n+1)$ with coefficients of $p_{\mathcal{Z}}^{efg\pm}$ as its column vectors.
If $\Sigma\subset \mathbb R$, set $\mathcal M_\mathbb R := [p_{\mathcal{Z}}^{eg}]$.
\begin{proposition}\label{detprop}
We have
\[
|\det(\mathcal M)|=2^{-\frac{n+1}{2}}\prod_{ijk} \frac{1}{|\Im(z_{ijk})|^{1/2}}\prod_{ijk< i'j'k'} |z_{ijk}-z_{i'j'k'}|,
\]
where $0\leq i,i'<2M, 0\leq j,j'< 2M,k'\le n_{i'j'}$ and $k\leq n_{ij}$.
If $\Sigma\subset \mathbb R$, $$|\det(\mathcal M_\mathbb R)| = \prod_{ik<i'k'}|z_{ik}-z_{i'k'}|$$ where $0\le i,i'<2M, k\le n_i$, and $k'\le n_{i'}$.
\end{proposition}
\begin{proof}
 Let 
\[
\mathcal N:=\left[\frac{p_{\mathcal{Z}}(x)}{(x-z_{ijk})}\right],
\]
where $0\leq i<2M, 0\leq j < M,$ and $k\leq n_{ij}$, be the square matrix of size $(n+1)\times (n+1)$ with the coefficients of $\frac{p_{\mathcal{Z}}(x)}{(x-z_{ijk})}$ as its column vectors.
By \eqref{pmdef} for fixed indices $efg$, we have 
\[
\left[p_{\mathcal{Z}}^{efg+},p_{\mathcal{Z}}^{efg-}  \right]=\left[\frac{p_{\mathcal{Z}}(x)}{(x-z_{efg})},\frac{p_{\mathcal{Z}}(x)}{(x-\overline{z_{efg}})}\right]
\begin{bmatrix}1/2 & \frac{1}{2i\Im(z_{efg})}  \\ 1/2  & -\frac{1}{2i\Im(z_{efg})} \end{bmatrix}.
\]
Hence,
\begin{equation}\label{detA}
    |\det(\mathcal M)|=|\det(\mathcal N)|2^{-\frac{n+1}{2}}\prod_{ijk} \frac{1}{|\Im(z_{ijk})|^{1/2}}
\end{equation} where $0\leq i<2M, 0\leq j< 2M, k\leq n_{ij}.$
Let
\[
V:=[z_{ijk}^e]
\]
be the Vandermonde matrix of size $(n+1)\times (n+1),$ where $0\leq e<n+1.$ It is well-known that
\[
|\det V|=\prod_{ijk}\prod_{ijk< i'j'k'} |z_{ijk}-z_{i'j'k'}|.
\]
We have 
\[
V\mathcal N=\diag\left[\prod_{ijk\neq efg} (z_{efg}-z_{ijk})\right].
\]
Hence,
\[
|\det(\mathcal N)|=\prod_{ijk}\prod_{ijk< i'j'k'} |z_{ijk}-z_{i'j'k'}|.
\]
Therefore by \eqref{detA},
\[
|\det(\mathcal M)|=2^{-\frac{n+1}{2}}\prod_{ijk} \frac{1}{|\Im(z_{ijk})|^{1/2}}\prod_{ijk< i'j'k'} |z_{ijk}-z_{i'j'k'}|.
\]
Now assume $\Sigma\subset\mathbb R$, then $V_\mathbb R:=[z_{ik}^e]$ satisfies $|\det V_\mathbb R|=\prod_{ik}\prod_{ik<i'k'}|z_{ik}-z_{i'k'}|$. Therefore,
$$V_\mathbb R\mathcal M_\mathbb R = \text{diag}\left[\prod_{ik\neq i'k'}(z_{ik}-z_{i'k'})\right].$$
Hence we have that $|\det(\mathcal M_\mathbb R)| = \prod_{ik<i'k'}|z_{ik}-z_{i'k'}|$ as desired.
\end{proof}

\begin{proposition}\label{vol_K_n}
We have
\[
\vol(K_n) \gg n^{-cn^{2-\frac{\delta}{6}}}e^{\frac{n^2}{2}I(\mu)}
\]
where $c$ is a constant dependent only on $\Sigma$, $\mu$, and $A$ from \eqref{Bij}.
If $\Sigma\subset\mathbb R$, we have
$$\vol(K_n) \gg e^{-cn^{2-\frac{\delta}{6}}}e^{\frac{n^2}{2}I(\mu)}$$ for some constant $c$ dependent only on $\Sigma$, $\mu$, and $A$ from \eqref{Bireal}.
\end{proposition}
\begin{proof}
 By Proposition~\ref{simplex}, $e^{Cn^{1-\frac{\delta}{6}}\log(n)}K_n$ includes the simplex with vertices $\left\{0, p_{\mathcal{Z}}^{ijk\pm} \right\}$ with volume $\frac{|\det(A)|}{(n+1)!}$, where $0\leq i<2M, 0\leq j\leq M-1, k\leq n_{ij}.$ Since $K_n$ is symmetric, there are $2^{n+1}$ simplexes with the same volume and disjoint interior inside $e^{Cn^{1-\frac{\delta}{6}}\log(n)}K_n$ by choosing different signs for the $n+1$ vertices $\left\{0,\pm p_{\mathcal{Z}}^{ijk\pm} \right\}$.
Hence,
\[
\vol(K_n) \geq \frac{2^{n+1}}{(n+1)!}e^{-Cn^{2-\frac{\delta}{6}}\log(n)}|\det(\mathcal M)|.
\]
By Proposition~\ref{detprop},
\[
\vol(K_n) \geq \frac{2^{\frac{n+1}{2}}}{(n+1)!}e^{-Cn^{2-\frac{\delta}{6}}\log(n)}\prod_{ijk} \frac{1}{|\Im(z_{ijk})|^{1/2}}\prod_{ijk< i'j'k'} |z_{ijk}-z_{i'j'k'}|.
\]
We have that $|\Im(z_{ijk})| \le c$ where $\text{supp}(\mu)\in [a,b]\times[-c,c]$ with $c\ge 1$.
 Therefore
$$\prod_{ijk}\frac{1}{|\Im(z_{ijk})|^{1/2}}\prod_{ijk<i'j'k'}|z_{ijk}-z_{i'j'k'}|\ge c^{-n/2}\prod_{ijk<i'j'k'}|z_{ijk}-z_{i'j'k'}|.$$
From \eqref{maineq} using $M=\lfloor n^\frac{1}{3}\rfloor$ and $L=\lfloor n^{\frac{1}{3}-\frac{\delta}{6}}\rfloor$,
we have $\prod_{ijk<i'j'k'}|z_{ijk}-z_{i'j'k'}|= e^{\frac{n^2}{2}I(\mu)+O(n^{2-\frac{\delta}{6}}\log(n))}$ and so for some (potentially negative) constant $C'$, this is at least $e^{\frac{n^2}{2}I(\mu)+C'n^{2-\frac{\delta}{6}}\log(n)}$. This gives
$$\vol(K_n)\gg \frac{2^\frac{n+1}{2}}{(n+1)!c^{n/2}}e^{\frac{n^2}{2}I(\mu)+(C'-C)n^{2-\frac{\delta}{6}}\log(n)}\gg
\frac{n^{(C'-C)n^{2-\frac{\delta}{6}}}e^{\frac{n^2}{2}I(\mu)}}{\left(n+1\right)^{n+1}c^n}\gg n^{\tilde C n^{2-\frac{\delta}{6}}}e^{\frac{n^2}{2}I(\mu)}$$
since $\frac{(n+1)^{n+1}}{n^n}\sim (n+1)e$ for some constant $\tilde C$ only dependent on $A,\Sigma$, and $\mu$.
\newline

If $\Sigma\subset\mathbb R,$ we get
$$\text{vol}(K_n) \ge \frac{2^{n+1}}{(n+1)!}e^{-Cn^{2-\frac{\delta}{6}}}|\det(\mathcal M_\mathbb R)|.$$
By Proposition \ref{detprop} and \eqref{maineqreal},
$$|\det(\mathcal M_\mathbb R)|=\prod_{ik<i'k'}|z_{ik}-z_{i'k'}| = \frac{n^2}{2}I(\mu) + O(n^{2-\delta/6})$$
choosing $M=\lfloor n^{1/3}\rfloor$ and $L=\lfloor n^{1/3-\epsilon}\rfloor$. Therefore, $|\det(\mathcal M_\mathbb R)|=\frac{n^2}{2}I(\mu)+C'n^{2-\delta/6}$ for some constant $C'$ dependent only on $A,\Sigma$, and $\mu$.
Therefore,
$$\text{vol}(K_n)\gg \frac{2^{n+1}}{(n+1)!}e^{\frac{n^2}{2}I(\mu)+(C'-C)n^{2-\frac{\delta}{6}}} \gg e^{\frac{n^2}{2}I(\mu)+\tilde Cn^{2-\frac{\delta}{6}}}$$
for some constant $\tilde C$ dependent only on $A,\Sigma$, and $\mu$.
\end{proof}

\begin{proposition}\label{mink}
Suppose that $\mu \in \mathcal{B}_{\Sigma}.$ There exists an integral polynomial $p_0$ with $\deg(p_0)\leq n$ such that
\[
\frac{\log |p_0(x)|}{n} \leq U_{\mu}(x)-\frac{I(\mu)}{2}+C\log(n)n^{-\delta/6},
\]
where $C$ is a constant independent of $n.$
Moreover, 
\[
I(\mu)\geq 0.
\]
If $\Sigma\subset\mathbb R$, there exists an integral polynomial $p_0$ with $\deg(p_0)\le n$ such that
\[
\frac{\log|p_0(x)|}{n} \le U_\mu(x) - \frac{I(\mu)}{2} + Cn^{-\delta/6}
\]
for some constant $C$ independent of $n$.
Consequently, $I(\mu)\ge 0$.
\end{proposition}
\begin{proof}
By Proposition \ref{vol_K_n} there are constants $c$
such that for all $n>0$, $\vol(K_n) \ge n^{cn^{2-\frac{\delta}{6}}}e^{\frac{n^2}{2}I(\mu)}$. 
By Minkowski's second theorem~\eqref{minksec}, there exists $p_0\in\mathbb Z[x]\cap \lambda_0K_n$  of degree at most $n$, where
\begin{equation}\label{upperb}
\lambda_{0}\leq 2 \left( \vol(K_n)  \right)^{-1/(n+1)} \ll n^{cn^{1-\frac{\delta}{6}}}e^{-\frac{n}{2}I(\mu)}.     
\end{equation}
By definition of $K_n,$
\[
\frac{\log |p_0(x)|}{n} \leq  U_{\mu}(x)+ \frac{\log(\lambda_0)}{n}
\]
By using~\eqref{upperb}, we obtain
\[
\frac{\log |p_0(x)|}{n} \leq U_{\mu}(x)-\frac{I(\mu)}{2}+c\log(n)n^{-\frac{\delta}{6}}.
\]
By taking the expected value of the above inequality, we have
\[
0\leq \int \frac{\log |p_0(x)|}{n} d\mu(x) \leq \frac{I(\mu)}{2}+c\log(n)n^{-\frac{\delta}{6}},
\]
where we used the assumption $\mu \in \mathcal{B}_{\Sigma}$ for the first inequality. By letting $n\to \infty$, we have 
\[
I(\mu)\geq 0.
\]
For $\Sigma\subset\mathbb R$, the same calculation gives the desired result except \eqref{upperb} becomes
$$\lambda_0 \ll e^{cn^{1-\frac{\delta}{6}}}e^{-\frac{n}{2}I(\mu)}.$$
Carrying this through the computation,
$$\frac{\log|p_0(x)|}{n} \le U_\mu(x) - \frac{I(\mu)}{2} + cn^{-\frac{\delta}{6}}.$$
Integrating, we get $0\le \frac{I(\mu)}{2}+cn^{-\frac{\delta}{6}}$. As before, taking $n\to\infty$ shows $I(\mu)\ge 0$.
\end{proof}
\subsection{Proof of Theorem~\ref{mdim}}
Suppose that $\mu$ is a probability measure with compact support on the complex plane and that for every non-zero  $Q(x)\in \mathbb{Z}[x]$,
\[
 \int \log |Q(z)| d\mu(z) \geq 0.  
\]

First, we reduce the proof of Theorem~\ref{mdim} to the case of H\"older probability measures by the following proposition.
\begin{proposition}\label{holderreduction}
Suppose that $\Sigma\subset \mathbb{C}$ is compact  and $\mu\in \mathcal{B}_{\Sigma}.$  For any $\rho>0,$  there exists a sequence of H\"older probability measures $\mu_n\in\mathcal{B}_{\Sigma(\rho)}$ such that   \( \lim_{n\to \infty} \mu_n=\mu\). Furthermore, if $\Sigma\subset \mathbb{R}$ is compact  and $\mu\in \mathcal{B}_{\Sigma}$, exists a sequence of H\"older probability measures $\mu_n\in\mathcal{B}_{\Sigma_{\mathbb{R}}(\rho)}$ such that   \( \lim_{n\to \infty} \mu_n=\mu\). 
\end{proposition}

\begin{proof}
    Let $\lambda_{r}$ be the uniform probability measure on the circle  of radius $r$ centered at the origin.  Let $\mu*\lambda_r$ be the convolution of $\mu$ with $\lambda_r.$  We show that $U_{\mu*\lambda_r}(z)\geq U_{\mu}(z)$ for every $z\in \mathbb{C}$ and every $r\geq 0.$ In fact, we
    have 
    \[
         U_{\mu*\lambda_r}(z)-U_{\mu}(z)=\int \left( \int_{0}^1\log|z-x+re^{2\pi i \theta}|- \log|z-x| d\theta \right)d\mu(x)\geq 0,
    \]
where we used Jensen's formula~\eqref{jensen} to show the inner integral is positive. Therefore, $\mu*\lambda_r \in \mathcal{B}_{\Sigma(\rho)}$ for any $r\leq \rho.$ 
    Define
    \[
    \mu_n:=\mu*\lambda_{\rho/n}.
    \]
    It is clear from the definition that \( \lim_{n\to \infty} \mu_n=\mu\) and $\mu_n$ is a H\"older measure with exponent at least 1. This completes the poof of our proposition for $\Sigma\subset \mathbb{C}$. By compactness of $\Sigma\subset \mathbb{R}$, it follows that $\Sigma(\rho)$
    is a finite union of intervals. By \cite[Proposition 2.5]{Smith}, there exists a sequence probability measures $\mu_n\in\mathcal{B}_{\Sigma_{\mathbb{R}}(\rho)}$ with H\"older exponent at least $1/2$ such that   \( \lim_{n\to \infty} \mu_n=\mu\).
\end{proof}
\begin{proof}[Proof of Theorem~\ref{mdim}]

Part~\eqref{listcond1} is the special case of part~\eqref{listcond2} for $m=1.$ So, it is enough to show that part~\eqref{listcond1} implies part~\eqref{listcond2}. Suppose that $\mu$ satisfies the conditions of part~\eqref{listcond1}, which is equivalent to $\mu\in \mathcal{B}_{\Sigma}$ by definition.
\newline

Next, we reduce the theorem to the case where $\mu$ is H\"older.  By Proposition~\ref{holderreduction}, there exists a sequence of H\"older probability measure $\mu_n\in\mathcal{B}_{\Sigma(\rho)}$ such that   \( \lim_{n\to \infty} \mu_n=\mu\).  Suppose that Theorem~\ref{mdim} holds for H\"older probability measures. Then
\[
\int \log |Q(z_1,\dots,z_m)|d\mu_n(z_1)\dots d
\mu_n(z_m)\geq0,
\]
for every $n.$ Note that $\log |Q(z_1,\dots,z_m)|$ is an upper semi-continuous function. Hence, by monotone convergence theorem, we have 
\[
\int \log |Q(z_1,\dots,z_m)|d\mu(z_1)\dots d
\mu(z_m) \geq \limsup_{n\to \infty }\int \log |Q(z_1,\dots,z_m)|d\mu_n(z_1)\dots d
\mu_n(z_m)\geq 0.
\]
So without loss of generality we assume that $\mu\in \mathcal{B}_{\Sigma}$ is H\"older. The proof is by induction on $m,$ and part~\eqref{listcond1} is the base of our induction hypothesis.  We assume that part~\eqref{listcond2} holds for $m$ and we show that it holds for $m+1.$ Suppose that $Q(x_1,\dots,x_{m+1})\in \mathbb{Z}[x_1,\dots,x_{m+1}].$ Without loss of generality, we assume that $Q$ is irreducible in $\mathbb{Z}[x_1,\dots,x_{m+1}]$ and does not belong to $\mathbb{Z}[x_{m+1}].$ Let
\[
U_{Q}(z_{m+1}):=\int \log |Q(z_1,\dots,z_m,z_{m+1})|d\mu(z_1)\dots d\mu(z_m).
\]
It is enough to show that
\begin{equation}\label{post}
\int U_{Q}(z) d\mu(z)\geq 0.    
\end{equation}
We fix $z_i\in \mathbb{C}$ for $1\leq i\leq m$ and consider $Q$ as a polynomial in $x_{m+1}$ variable and assume that it has degree $d$ in $x_{m+1}.$ Then
\[
Q(z_1,\dots,z_m,x_{m+1})=\sum_{i=0}^d a_i(z_1\dots,z_m)x_{m+1}^i=a_d(z_1,\dots,z_m)\prod_{i=1}^d (x_{m+1}-\xi_i(z_1,\dots,z_{m}))
\]
where $\xi_i(z_1,\dots,z_{m}) \in \mathbb{C}$ are the complex roots of $Q$. By our induction hypothesis, we have 
\[
\int \log |a_d(z_1,\dots,z_m)|d\mu(z_1)\dots d\mu(z_m)\geq 0.
\]
This implies that $a_d(z_1,\dots,z_m)\neq 0$ for almost all $z_1,\dots,z_m$ with respect to product measure $\mu\times\dots\times \mu$.
We define the measure $\mu_Q$ on $\mathbb{C}$ as follows
\begin{equation}\label{defmuq}
    \mu_Q:= \int \sum_{i} \delta_{\xi_i(z_1,\dots,z_{m})} d\mu(z_1)\dots d\mu(z_m)
\end{equation}
where $\delta_{a}$ is the delta probability measure at point $a\in \mathbb{C}.$ We have
\[
U_Q(z)=a_Q+ \int \log|z-x|d\mu_{Q}(x)
\]
where
\begin{equation}\label{defaq}
a_Q:=\int \log|a_d(z_1,\dots,z_m)|d\mu(z_1)\dots d\mu(z_m).    
\end{equation}
Note that by the induction hypothesis, we have 
\[
a_Q\geq 0.
\]
By Proposition~\ref{mink}, there exists an integral polynomial $P_0(x)$ with $\deg(P_0)\leq n$ such that
\[
\frac{\log |P_0(x)|}{n} \leq \int\log |z-x| d\mu(z)+Cn^{-\delta}.
\]
We take the expected value of the above inequality with respect to $d\mu_Q$ and obtain 
\[
a_Q+\int  \frac{\log|P_0(x)|}{n} d\mu_Q(x) \leq a_Q+\int \log|z-x| d\mu(z)d\mu_{Q}(x) +\frac{Cd}{n^{\delta}}=\int U_{Q}(z) d\mu(z)+\frac{Cd}{n^{\delta}}.
\]
It is enough to show that 
\begin{equation}\label{pn}
na_Q+\int \log |P_0(x)| d\mu_Q(x) \geq 0,
\end{equation}
since \eqref{post} follows by letting $n\to \infty.$
In fact, we prove that for every integral polynomial $P(x),$
\begin{equation}\label{pns}
\deg(P)a_Q+\int \log |P(x)| d\mu_Q(x) \geq 0.    
\end{equation}
This implies \eqref{pn}, since $\deg(P_0)\leq n$ and $a_Q\geq0.$
For proving~\eqref{pns}, we consider $P$ as an integral polynomial in $x_{m+1}$. Define
\[
R(Q,P)(x_1,\dots,x_m):= Res(Q(x_1,\dots,x_{m+1}),P(x_{m+1}))\in \mathbb{Z}[x_1,\dots,x_m]
\]
to be the resultant of $Q$ and $P$ as polynomials in the variable $x_{m+1}$. By our assumption, $Q$ is irreducible and does not divide $P(x_{m+1}).$ Hence, $Res(Q(x_1,\dots,x_{m+1}),P(x_{m+1}))\neq 0.$
By our induction hypothesis,
\[
\int \log R(Q,P)(z_1,\dots,z_m)d\mu(z_1)\dots d\mu(z_m)\geq 0.
\]
We note that
\[
\log R(Q,P)(z_1,\dots,z_m)=\deg(P)\log|a_d(z_1,\dots,z_m)|  + \sum_{i=1}^d \log|P(\xi_i(z_1,\dots,z_{m}))|.
\]
Hence,
\[
\int \deg(P)\log|a_d(z_1,\dots,z_m)|  + \sum_{i=1}^d \log|P(\xi_i(z_1,\dots,z_{m}))|d\mu(z_1)\dots d\mu(z_m)\geq 0.
\]
The above implies~\eqref{pns} by definitions of $\mu_Q$ and $a_Q$ in \eqref{defmuq} and \eqref{defaq}. This completes our proof and implies 
\[
\int \log |Q(z_1,\dots,z_m,z_{m+1})|d\mu(z_1)\dots d
\mu(z_m)d\mu(z_{m+1})\geq0.
\]
\end{proof}

 \subsection{Lower bound on the product of Minkowski's minima factors}
 First, we give a lower bound on the product of the Minkowski minima factors by using Theorem~\ref{mdim}.
\begin{proposition}\label{lowerd}
We have 
\[
\prod_{i=0}^{m-1} \lambda_i \geq \frac{1}{m!} e^{(-nm+\frac{(m-1)m}{2})I_\mu}.
\]
\end{proposition}
\begin{proof}
By definition of $\{\lambda_i: 0\leq i\leq m-1\}$, there exists linearly independent integral polynomials $\{P_0,\dots,P_{m-1}\}$, where 
\[
|P_i(z)|\leq \lambda_i e^{nU_\mu(z)}.
\]
for every $z\in \mathbb{C}.$ Let $\{z_1,\dots,z_m \}\subset \mathbb{C}.$
Therefore,
\begin{equation}\label{detinq}
\det [P_i(z_j)]\leq m!\prod_{i=0}^{m-1} \lambda_i e^{nU_\mu(z_i)}.    
\end{equation}
On the other hand, we have 
\[
\det [P_i(x_j)]=Q_m(x_1,\dots,x_m)\prod_{i<j} (x_i-x_j)
\]
for some $Q_m\in \mathbb{Z}[x_1,\dots,x_m].$  By taking the average of the above, we obtain
\[
\int \log |\det [P_i(z_j)]| d\mu(z_1)\dots d\mu(z_m)=\frac{(m-1)m}{2}I_{\mu}+\int \log |Q_m(z_1,\dots,z_m)| d\mu(z_1)\dots d\mu(z_m),
\]
where 
\[
I(\mu):=\int \log|z_1-z_2|d\mu(z_1)d\mu(z_2).
\]
By the second part of Theorem~\eqref{mdim}, we have 
\[
\int \log |Q_m(z_1,\dots,z_m)| d\mu(z_1)\dots d\mu(z_m)\geq 0.
\]
Hence,
\[
\int \log |\det [P_i(z_j)]| d\mu(z_1)\dots d\mu(z_m)\geq \frac{(m+1)m}{2}I_{\mu}.
\]
By~\eqref{detinq}, we obtain 
\[
mnI_{\mu}+\sum_{i} \log |\lambda_i|+ \log (m!)\geq\int \log |\det [P_i(z_j)]| d\mu(z_1)\dots d\mu(z_m)\geq \frac{(m+1)m}{2}I_{\mu}.
\]
Therefore,
\[
\prod_i \lambda_i \geq \frac{1}{m!} e^{(-nm+\frac{(m-1)m}{2})I_\mu}.
\]
\end{proof}

\begin{proposition}\label{upperbound}
We have
\[
\lambda_n \leq   n^{cn^{1-\frac{\delta}{12}}}. 
\]
If $\Sigma\subset \mathbb R$, then $\lambda_n \le e^{cn^{1-\frac{\delta}{12}}}$.
\end{proposition}

\begin{proof}
Let $m=n-n^{\delta_1}$ where $0<\delta_1<1$ is specified later. By Proposition~\ref{lowerd},
\[
\prod_{i=0}^m \lambda_i \geq \frac{1}{m!} e^{-nm+\frac{(m-1)m}{2}I_\mu}.
\]
By Minkowski's second theorem and Proposition~\ref{vol_K_n},
\[
 \lambda_0\lambda_1\dots \lambda_n\leq 2^{n+1}\vol(K_n)^{-1} \ll2^{n+1} n^{cn^{2-\frac{\delta}{6}}}e^{-\frac{n^2}{2}I(\mu)}.
\]
By combining the above inequalities, we have
\[
\prod_{i=m+1}^n \lambda_i \ll n^{cn^{2-\frac{\delta}{6}}} e^{-\frac{(n-m)^2}{2}I(\mu)}.
\]
By Proposition~\ref{mink}, we have $I(\mu)\geq 0$ and
\[
\prod_{i=m+1}^n \lambda_i \ll n^{cn^{2-\frac{\delta}{6}}}.
\]
Hence, by taking $\delta_1=1-\delta/12,$ we obtain
\[
\lambda_{m+1}\leq n^{cn^{2-\frac{\delta}{6}-\delta_1}} \leq n^{cn^{1-\frac{\delta}{12}}}.
\]
By Proposition~\ref{successive}, we have
\[
\lambda_n \leq \lambda_m (An)^{C(n-m)} \leq  n^{cn^{1-\frac{\delta}{12}}}. 
\]

This completes the proof of our proposition. Note that we may change constant $c$ to a larger constant in the above lines of argument. 
\end{proof}

\subsection{Equilibrium measure}\label{chebsec}
In this subsection, we apply the results of our previous section to the equilibrium measure and prove Proposition~\ref{chpol}. Our proof is based on Fekete and Szeg\"o~\cite[Theorem D]{MR72941} which is also used by  Robinson~\cite{Robinson} and Smith~\cite[Proposition 4.1]{Smith} to prove a similar result. All these results are essential in controlling the roots to remain inside $\Sigma(\rho).$
\\

Suppose that $\Sigma$ is compact, is invariant under complex conjugation, and its complement is open and connected. Let $D$ be any open set containing $\Sigma$. Recall that
\[
\Sigma(\rho):=\{z\in \mathbb{C}: |z-\sigma|<\rho \text{ for some }\sigma\in\Sigma \}.
\]
Note that for any open set $D$ containing $\Sigma,$ there exist $\rho>0$ such that $\Sigma(\rho)\subset D.$

Let $\mu_{eq}$ be the equilibrium measure of $\Sigma.$ Fix any $\rho>0$ and $z_0\in \Sigma.$ It is well known that 
\begin{equation}\label{eqbd}
    U_{\mu_{eq}}(z)\geq \log (d_{\Sigma})+\delta_{\rho,\Sigma}
\end{equation}

for any $z\notin\Sigma(\rho),$ where $\delta_{\rho,\Sigma}>0 $ is a constant that only depends on $\Sigma$ and $\rho.$  Let $p_{\mathcal{Z}}(x)$ be the polynomial constructed in~\eqref{measpol} associated to the equilibrium measure.
By~\eqref{potapp} and~\eqref{eqbd} there exists some $R>d_{\Sigma}$ and $\mathcal Z$ with $|\mathcal Z|=n_0$ so that
\[
p_{\mathcal{Z}}(z)\geq R^{n_0}
\]
for any $z\notin\Sigma(\rho).$
Fix $z_0\in \Sigma(\rho)$ and $p_{\mathcal{Z}}(x),$ and  define
\begin{equation}\label{defTm}
    T_m(x):=(x-z_0)^rp_{\mathcal{Z}}(x)^q,
\end{equation}
where $m=n_0q+r$ is the Euclidean division of $m$ by $n_0$.  
Let $m\geq h$ be two positive integers. It follows that 
    \begin{equation}\label{cheby}
     T_m(z)\geq CR^{m-h}T_h(z)   
    \end{equation}
for any $z\notin\Sigma(\rho),$  where $C$ only depends on fixed parameters $\rho$ and $n_0.$

\begin{proposition}\label{chpol}
    Given a compact subset    $\Sigma\subset \mathbb{C}$ invariant under complex conjugation with connected complement,
    a monic polynomial with real coefficients $p(x)$, and $\rho>0$, there exists fixed integers $l_0$ and $k_0$ independent of $p(x)$ and a monic polynomial $Q_m(x)$ with $\deg(Q_m)=m$ for every $m=kl_0$ where $k>k_0$  such that 
    \[
    |Q_m(z)|\geq \kappa^m
    \]
    for every $z\notin \Sigma(\rho)$ where  $\kappa>d_{\Sigma}.$ Moreover, the coefficient of $x^i$ in $p(x)Q_m(x)$  is an even integer for every $\deg(p)\leq i\leq \deg(p)+m-1$, and all complex roots of $Q_m(x)$ are inside $\Sigma(\rho).$
\end{proposition}
\begin{proof}
Recall $T_m(x)$ in~\eqref{defTm} for $m\in \mathbb{Z},$ and define
\[
q_{k,l}(x):=\left(T_k(x)+a_1T_{k-1}(x)+\dots+a_kT_0(x) \right)^l
\]
where $|a_i|\leq \frac{1}{l}$ are chosen recursively for $1\leq i\leq k$ such that $q_{k,l}(x)p(x)$ has even coefficients for $x^{\deg(p)+\deg(q_{k,l})-i},$ where $1\leq i\leq k.$ Note that this is possible since the coefficient of $x^{\deg(p)+\deg(q_{k,l})-j}$ for $j\leq l$ in  $q_{k,l}(x)p(x)$ is given by
\[
la_j+F_j(a_1,\dots,a_{j-1})
\]
where $F_j(a_1,\dots,a_{j-1})$ is a polynomial in terms of $a_i$ for 
$i<j.$ By~\eqref{cheby}, we have
\begin{align*}
    |q_{k,l}(z)|&=\left|\left(T_k(z)+a_1T_{k-1}(z)+\dots+a_kT_0(z) \right)^l\right|
    \\
    &\geq \left|T_k(z)\left(1-\frac{C}{l}\sum_{j\geq 1}\frac{1}{R^j}\right)\right|^l \geq   |T_k(z)|^l\left(1-\frac{C}{l(R-1)}\right)^{l}.
\end{align*}
We fix an integer $l_0> \frac{C}{R-1}$, and let
\[
C':=\left(1-\frac{C}{l_0(R-1)}\right)^{l_0}>0.
\]
Hence,
\begin{equation}\label{fineq}
|q_{k,l_0}(z)| \geq C' T_k(z)^{l_0}\geq C''|p_{\mathcal{Z}}(z)|^{\left\lfloor\frac{k}{n_0}\right\rfloor l_0}
\end{equation}
where $C''$ is a fixed constant depending on our fixed parameters.
Finally, let $m:=l_0k$ and 
\[
Q_m(x):=q_{k,l_0}(x)-\sum_{i=0}^{m-k}b_iT_i(x),
\]
where $|b_i|\leq 1$ are chosen recursively for $0\leq i\leq m-k$ such that $Q_m(x)p(x)$ has even coefficients for $x^{j}$ and $\deg(p)\leq j\leq \deg(p)+ m-k.$ By \eqref{fineq}, \eqref{cheby}, and \eqref{defTm}, we have
\begin{align*}
|Q_m(z)|\geq |q_{k,l_0}(z)|-\sum_{i=0}^{m-k}|T_i(z)|\geq |p_{\mathcal{Z}}(z)|^\frac{kl_0}{n_0} \left( C''-\frac{C_1}{R^{k-1}(R-1)}  \right).
\end{align*}
Our proposition follows by taking $k>k_0$ where $k_0$ is a fixed constant.
\end{proof}
Suppose that $\Sigma \subset \mathbb{C}$ is compact and symmetric about the real axis and is a finite disjoint union of rectangles and real intervals. We write
\[
\Sigma=\bigcup_{i=1}^{N_1}(R_i \cup \bar{R_i}) \bigcup_{j=1}^{N_2}I_j
\]
where $R_i\subset \mathbb{C}$ are rectangles and $I_j\subset \mathbb{R}$ are intervals. We define
\[
\Sigma_{\mathbb{R}}(\rho):= \bigcup_{i=1}^{N_1}(R_i(\rho) \cup \bar{R_i}(\rho)) \bigcup_{j=1}^{N_2}{I_j}_{\mathbb{R}}(\rho)
\]
for any $\rho>0$. Note that if $\bigcup_{j}{I_j}_{\mathbb{R}}\neq \emptyset$, then $\Sigma_{\mathbb{R}}(\rho)\subsetneq \Sigma(\rho)$.
\\

For proving Corollary~\ref{gsmith} which generalizes Smith's result to $\mathbb{C}$, we use Proposition~\ref{chpolreal}. The statement of this proposition is motivated by Smith~\cite[Proposition 4.1]{Smith}.
However, Smith's method does not imply this proposition.  Smith's method is based on the work of Fekete and Szeg\"o~\cite[Theorem D]{MR72941} and   Robinson~\cite{Robinson}. Smith and Robinson used the properties of the Chebyshev's polynomials  from \cite{chebyshev_R} which is only proven for subsets of $\mathbb R$ and exploits the sign changes of the polynomials on real intervals. This does not work for complex sets, so our proof of Proposition~\ref{chpolreal} is conceptually different. Also, it allows us to keep the roots inside $\Sigma_{\mathbb{R}}(\rho)\subset \Sigma(\rho)$.

\begin{proposition}\label{chpolreal}
Suppose that $ \Sigma=\bigcup_i(R_i \cup \bar{R_i}) \bigcup_{j}I_j$ is compact and symmetric about the real axis and is a finite disjoint union of rectangles and real intervals and $\bigcup_{j}I_j\neq \emptyset$ with $d_{\Sigma}\geq 1$. Given a monic polynomial with real coefficients $p(x)$ of degree $n-m$, $\rho>0$, and an even integer $n^{\varepsilon}\leq m\leq n,$ there exists a monic polynomial with real coefficients  $Q_m(x)$ such that $\deg(Q_m)=m$ and each of the following hold.
    \begin{enumerate}
    \item The degree $i$  coefficient of the product $pQ_m$  is an even integer for  $n-m\leq i\leq n-1$.
    \item \label{2deform} Take $X$ to be the set of roots of $Q_m$. Then $X$ is a subset of $\Sigma_{\mathbb{R}}(\rho)$ and symmetric about the real axis. 
    \item \label{4deform}We have 
    \begin{equation*}
\frac{\log |Q_{m}(z)|}{m} =\frac{\min_{\alpha\in X}\log|z-\alpha|}{m}+  U_{\mu_{eq}}(z)+O(m^{-\frac{\delta}{8}}\log(m))
\end{equation*}
for  all $z\in \mathbb{C}$ and some $\delta>0,$ where $\mu_{eq}$ is the equilibrium measure for $\Sigma_{\mathbb{R}}(\rho/10).$

    \item~\label{3deform} Given any root $\alpha$ of $Q_m$ and any $\alpha'$ that is either a root of $p$  or a boundary point of $\Sigma_{\mathbb{R}}(\rho)$, we have
\(
n^{-3} \leq |\alpha-\alpha'|.
\)
\end{enumerate}

\end{proposition}

\begin{proof}
Let  $m_1=\lfloor m\mu_{eq}(\bigcup_{i}R_i(\rho/10))\rfloor$ and $m_2=m-2m_1$, where $\mu_{eq}$ is the equilibrium measure of $\Sigma_{\mathbb{R}}(\rho/10)$.
Let $x_1,\dots,x_{m}$ be as follows
\begin{enumerate}
    \item $x_i$ for $i=1,2,\dots, m_1$ are identically independently selected according to $\mu_{eq}$ conditional on $x_i\in \bigcup_{i}R_i(\rho/10)$ for $1\leq i\leq m_1$.
    \item $x_{i+m_1}:=\bar{x_{i}}$ for $1\leq i\leq m_1$.
    \item $x_{2m_1+i}$ for $i=1,2,\dots,m_2$ are identically independently selected according to $\mu_{eq}$ conditional on $x_i\in \bigcup_{j}I_j(\rho/10)$ for $2m_1< i\leq m$.
\end{enumerate}

We note that $\mu_{eq}$ is H\"older with exponent $\delta \ge 1/2$ on a finite union of intervals since if $[a,b]\subset \Sigma\subset\mathbb R$, then the density of $\mu_{eq,[a,b]}$ is larger than that of $\mu_{eq,\Sigma}$. Furthermore, $d\mu_{eq}=\frac{dx}{\pi\sqrt{(b-x)(x-a)}}$ on $[a,b]$.
Let $\mathcal{Z}:=\{x_1,\dots,x_m\}$ and $P_{\mathcal{Z}}$ be the polynomial constructed in~\eqref{measpolreal} with roots $x_1,\dots,x_m$.
With probability $1-o(1)$, we note that  $x_1,\dots,x_m$ satisfy conditions in~\eqref{n_ij}, \eqref{Bij}, and \eqref{Bireal} with   $M=m^{1/4}$, $L=(mM^{-\delta})^{\frac{1}{2}+\varepsilon}< m^{1-\varepsilon}M^{-2}$, and \(
n^{-2-\varepsilon} \leq |x_i-\alpha'|
\) for any $\alpha'$ that is either a root of $p$, any other $x_j\neq x_i$, or a boundary
point of $\Sigma_{\mathbb{R}}(\rho).$ 
This follows from a Borel-Cantelli type argument. Indeed, recall the definition of $B_{ij}$ from \eqref{n_ij} and  let $n_{ij}$ be the number of elements of $\mathcal{Z}:=\{x_1,\dots,x_m\}$ lying inside  $B_{ij}$. Fix $i$ and $j$ and define the following random variables
\[
y^{ij}_k=
\begin{cases}
    1 &\text{ if $x_k\in B_{ij}$}
    \\
    0 &\text{ otherwise.}
\end{cases}
\]
Let $p=\mu(B_{ij})$. We have $\text{Prob}(y^{ij}_k=1)=p$ and $\text{Prob}(y^{ij}_k=0)=1-p.$
Note that
\[
n_{ij}=\sum_{k}y^{ij}_k.
\]

By Markov's inequality, we have 
\[
\text{Prob}\left(|n_{ij}-\mu(B_{ij})m|\geq L\right)\leq
\frac{\text{E}\left((n_{ij}-\mu(B_{ij})m)^{2h} \right)}{L^{2h}} 
\]
for any $h\geq 0$.
Moreover,

\begin{align*}
\text{E}\left((n_{ij}-\mu(B_{ij})m)^{2h} \right)=\text{E}\left(\sum_{k}y^{ij}_k-\text{E}(y^{ij}_k)\right)^{2h}
\\
\ll
\sum_{k_1,\dots,k_h}\prod_{l=1}^{h}\text{E}\left(y^{ij}_{k_l}-\text{E}(y^{ij}_{k_l})\right)^{2}\ll m^{h}\left(p(1-p) \right)^h
\\
\ll m^{h}p^h=\left(\mu(B_{ij})m\right)^h.
\end{align*}
Hence,
\[
\text{Prob}\left(|n_{ij}-\mu(B_{ij})m|\geq L\right)\leq \frac{\left(\mu(B_{ij})m\right)^h}{L^{2h}}\leq L^{-2h(\frac{\varepsilon}{1+\varepsilon})}.
\]
where we used the inequality \(\mu(B_{ij})m \leq  mM^{-\delta}\) from \eqref{holder}. Therefore,
\begin{align*}
\sum_{i,j}\text{Prob}\left(|n_{ij}-\mu(B_{ij})m|\geq L\right)&\ll 
\sum_{i,j}L^{-2h(\frac{\varepsilon}{1+\varepsilon})} \ll M^2 L^{-2h(\frac{\varepsilon}{1+\varepsilon})}
\end{align*}
for any $h\geq 0$. Since $L\gg m^{1/4}$,  by letting $h\gg \frac{1}{\varepsilon}$, we conclude our claim that with probability $1-o(1)$,  $x_1,\dots,x_m$ satisfy conditions in~\eqref{n_ij}, \eqref{Bij}, and \eqref{Bireal} with   $M=m^{1/4}$, $L=(mM^{-\delta})^{\frac{1}{2}+\varepsilon}$, and \(
n^{-2-\varepsilon} \leq |x_i-\alpha'|
\) for any $\alpha'$ that is either a root of $p$, any other $x_j\neq x_i$, or a boundary
point of $\Sigma_{\mathbb{R}}(\rho).$
By~\eqref{potapp} and~\eqref{eqbd}, this polynomial satisfies all the conditions above except the first one. 
Our method is to deform the roots of $P_{\mathcal{Z}}$ with a greedy algorithm so that it satisfies all the properties of Proposition~\ref{chpolreal}. We start with the degree $n-1$ coefficient of $pP_{\mathcal{Z}}.$
We pick at least $\lceil 10/\rho\rceil$ of them which are symmetric by the real axis, say $x_1,\dots,x_{k_0},$ and deform each of them less than $\rho/10$ so that the degree $n-1$ coefficient of $pP_\mathcal Z$ becomes even and the roots remain in side $\Sigma_{\mathbb{R}}(\rho)$ and symmetric by the real axis.
We show that this is possible.
\newline

Let $Q_{m,0}(x):=P_{\mathcal{Z}}(x)=x^m+\sum_{i=1}^{m}a_ix^{m-i}.$   The coefficient of $x^{\deg(p)+m-j}$ for $j\leq l$ in  $Q_{m,0}(x)p(x)$ is given by
$a_j+F_j(a_1,\dots,a_{j-1})$
where $F_j(a_1,\dots,a_{j-1})$ is a polynomial in terms of $a_i$ for 
$i<j.$ We define $x_i(t):=x_i+t\frac{\rho}{10}$ for $i\leq k_0$ and call the deformed polynomial $Q_{m,0,t}.$
By the intermediate value theorem, there exists $|t|\leq 1$ such that the degree $n-1$ coefficient of $pQ_{m,0,t}$ becomes an even number and the roots remain inside $\Sigma_{\mathbb{R}}(\rho).$ Suppose that after deformation, condition~\eqref{3deform} is violated for $y_1,\dots,y_k$ which are symmetric by the real axis.
By \eqref{Bireal}, it is possible to pick $k$ other roots of $Q_{m,0}(x)$ say $y_1',\dots, y_k'$ which are symmetric by the real axis such that there are no other roots close to $y_i'$ within distance $n^{-2-\varepsilon}.$ We update $y_1$ and $y_1'$ by $y_1+t$ and $y_1'-t.$
By a covering argument, there exists $t\ll n^{-3}$ such that $y_1+t$ is away from other roots with distance at least $n^{-3}.$ Similarly, update $y_i, y_i'$ so that the  condition~\eqref{3deform} is satisfied. 
Denote the updated polynomial by $Q_{m,1}.$ Note that $Q_{m,1}$
satisfies all the conditions above except the first one, and 
the degree $n-1$ coefficient of $pQ_{m,1}$ is even.
\\

We construct $Q_{m,l}$ for $l\geq 1$ recursively as follows. Suppose that $Q_{m,l}$
satisfies all the conditions except the first, and the degree $i$ coefficient of the product $pQ_{m,l}$  is an even integer for  $n-l\leq i\leq n-1$, we want to make the coefficient of $n-l-1$ even as well without changing the higher degree coefficients by deforming the roots of $Q_{m,l}$ as we did in the previous paragraph. 
\\

Suppose that $l\leq \log(n)^2.$  Let  $l_1=\lfloor l\mu_{eq}(\bigcup_{i}R_i(\rho/10))\rfloor$ and $l_2=l-2l_1$.
We pick at least $n^{\varepsilon}$ distinct symmetric collections of subsets $\mathcal{Z}_{\sigma_i}:=\{x_{\sigma_i(1)},\dots,x_{\sigma_i(l)}\} \subset\mathcal{Z}=\{x_1,\dots,x_m \}$ with $l$ elements for $1\leq i\leq n^{\varepsilon}$,
where $\varepsilon>0$ is fixed and is arbitrary small as follows:
\begin{enumerate}
    \item $x_{\sigma_i(1)},\dots,x_{\sigma_i(l_1)}$ are selected uniformly among $x_1,\dots,x_{m_1}$, with probability $1-o(1)$ these points have not been deformed or selected in previous steps, and we assume this.
    \item $x_{\sigma_i(j+l_1)}=\bar{x}_{\sigma_i(j)}$ for every $1\leq j\leq l_1$.
    \item $x_{\sigma_i(2l_1+1)},\dots,x_{\sigma_i(l)}$ are selected uniformly among $x_{2m_1+1},\dots,x_{m}$, with probability $1-o(1)$ these points have not been deformed or selected  in previous steps, and we assume this.
\end{enumerate}
Fix $\mathcal{Z}_{\sigma_i}$ and consider $\vec{x}=[x_{\sigma_i(1)},\dots,x_{\sigma_i(l)}]$, and  
let
\[
s_k:=\sum_{j=1}^lx_{\sigma_i(j)}^k,
\]
and 
\[
e_k:=\sum_{1\leq j_1<\dots<j_k\leq l}x_{\sigma_i(j_1)}\dots x_{\sigma_i(j_k)},
\]
where $e_0=1$ and $s_0=l.$
Let 
\(\nabla s_k=k[x_{\sigma_i(j)}^{k-1}]\) be the gradient of $s_k.$ Let
\[
v_l:=\left[\frac{1}{\prod_{j\neq 1}(x_{\sigma_i(1)}-x_{\sigma_i(j)})},\frac{1}{\prod_{j\neq2}(x_{\sigma_i(2)}-x_{\sigma_i(j)})},\dots, \frac{1}{\prod_{j\neq l}(x_{\sigma_i(l)}-x_{\sigma_i(j)})}\right].
\]
Note that the coordinates of $v_l$ for $j > 2l_1$ are real, and for $j\le l_1$, $v_{l,j}=\overline{v_{l,j+l_1}}$. This implies that the perturbation of the roots in this direction remains symmetric.
It follows that $v_l$  is orthogonal to  $\nabla s_k$ for every $1\leq k\leq l-1$ and
\(
\left<v_l,\nabla s_l \right>=l.
\)
We note that 
\[
e_k=\frac{1}{k} \sum_{j=1}^k (-1)^{j-1}e_{k-j}s_j.
\]
This implies that $\left<v_l,\nabla e_k \right>=0$ for every $0\leq k\leq l-1$ and
\(
\left<v_l,\nabla e_l \right>=(-1)^{l-1}.
\) We note that with probability 1, by~\eqref{potapp} and~\eqref{eqbd}, the size of $v_l$ is exponentially small and we have
\[
|v_l|\leq d_{\Sigma_{\mathbb{R}}(\rho/10)}^{-l(1+O(l^{-\delta}))}
\]
for some $\delta>0$.
We deform $\vec{x}$ or equivalently $\mathcal{Z}_{\sigma_i}$ with the following ODE
$\frac{d\vec{x}(t)}{dt}=v_l(t)$ and the initial condition $\vec{x}(0)=\vec{x}.$ We note that the associated subset $\mathcal{Z}_{\sigma_i}(t)$ deforms continuously with $t$ and it remains symmetric by the $x$ axis. By the intermediate value theorem for any large $A>0$, it is possible to deform $n^{\varepsilon }$ distinct $l$-tuples $[x_1,\dots,x_l]$ along the direction of $v_l$ for some $t<l^{-A}$  and make the coefficient of $x^{n-l-1}$   in $pQ_{m,l+1}$ is even.
It is clear that the perturbation  remains inside $\Sigma_{\mathbb R}(\rho).$
 Suppose that after deformation, condition~\eqref{3deform} is violated for $y_1,\dots,y_k,$ where $k\leq n^{\varepsilon}l$.
 By \eqref{Bireal}, it is possible to pick $kl\leq n^{\varepsilon}\log(n)^4$ other roots of $Q_{m,l}(x)$ say $y_{i,j}$ for $1\leq i\leq k$ and $1\leq j\leq l$ such that $y_{i,j}$ are i.i.d. according to $\mu_{eq}$. We deform $l+1$ tuple $\vec{x}_i:=\left[y_i,y_{i,1},\dots, y_{i,l}\right]$ with the following ODE
$\frac{d\vec{x}_i(t)}{dt}=v_{l+1,i}(t)$ and the initial condition $\vec{x}_i(0)=\vec{x}_i$ where
\[
v_{i,l+1}(t):=\left[\frac{1}{\prod_j(y_i-y_{1,j})},\frac{1}{(y_{i,1}-y_i)\prod_j(y_{i,1}-y_{i,j})},\dots, \frac{1}{(y_{i,l}-y_i)\prod_j(y_{i,l}-y_{i,j})}\right].
\]
 By a covering argument, there exists $t\ll n^{-3+\epsilon}$ such that $x_1(t)$ is away from other roots with distance at least $n^{-3}$.
Denote the updated polynomial by $Q_{m,l+1}.$ Note that $Q_{m,l+1}$
satisfies all the conditions above except the first one, and  the degree $i$  coefficient of the product $pQ_{m,l+1}$  is an even integer for  $n-l-1\leq i\leq n-1$.
\\

Finally, suppose that $m \geq l\geq \log(n)^2.$ We pick a symmetric  subsets $\mathcal{Z}_{\sigma}:=\{x_{\sigma(1)},\dots,x_{\sigma(l)}\} \subset\mathcal{Z}=\{x_1,\dots,x_m \}$ with $l$ elements that satisfies three conditions listed above. As before, let $\vec{x}=[x_{\sigma(1)},\dots,x_{\sigma(l)}]$
and 
\[
v_l:=\left[\frac{1}{\prod_{i\neq 1}(x_{\sigma(1)}-x_{\sigma(i)})},\frac{1}{\prod_{i\neq2}(x_{\sigma(2)}-x_{\sigma(i)})},\dots, \frac{1}{\prod_{i\neq l}(x_{\sigma{(l)}}-x_{\sigma{(i)}})}\right].
\]
By~\eqref{potapp} and~\eqref{eqbd},   the size of $v_l$ is smaller than any negative power of $n$
\[
|v_l|\leq d_\Sigma^{-l}\leq e^{-\log(n)^2}\ll n^{-A}
\]
for any $A>0.$
We deform $\vec{x}$ with the following ODE
$\frac{d\vec{x}(t)}{dt}=v_l(t)$ and the initial condition $\vec{x}(0)=\vec{x}.$ By the intermediate value theorem, there exists $|t|\leq 1$ so that the degree $n-l-1$ coefficient of $pQ_{m,l+1}$ is even. This time the perturbation does not change conditions \eqref{2deform}, \eqref{4deform} and \eqref{3deform},
since the perturbation is smaller than $n^{-A}$ for any $A$, and  the degree $i$  coefficient of the product $pQ_{m,l+1}$  is an even integer for  $n-m\leq i\leq n-1$. This completes the proof of our proposition. 
\end{proof}

\section{Proofs of main theorems}
\subsection{Proof of Theorem~\ref{main1}}  
In this subsection, we assume that  $\Sigma\subset \mathbb{C}$ is compact and  has empty interior with connected complement.  We also assume that $\mu \in \mathcal{B}_{\Sigma}$ is a H\"older measure.

\begin{proof}[Proof of Theorem~\ref{main1}]\label{rouche}
Suppose that $\mu \in \mathcal{B}_{\Sigma}$ is a H\"older measure with exponent $\delta,$  and $D$ is any open set containing $\Sigma$. Fix $\rho>0$ such that $\Sigma(\rho)\subset D.$
We apply Proposition~\ref{chpol} to $p_{\mathcal Z}(x)$ where $|\mathcal Z|=n-m$ as defined in~\eqref{measpol} and $m:= n^{1-\delta_0}$ for some $\delta_0>0$ that we specify later and obtain $Q_m(x)$ such that
\[
p_\mathcal Z(x)Q_m(x)=x^{n}+\sum_{i=0}^{n-1} a_ix^i
\]
where $a_i$ are even for $n-m\leq i\leq n-1.$ Let 
\[
r(x):=\frac{1}{2}+\sum_{i=0}^{n-1-m} \frac{a_i}{4} x^i.
\]We write $r(x)$ in terms of  linearly independent integral polynomials $\{P_0,\dots,P_{n-1-m}\}$ obtained from Minkowski's successive minima of $K_{n-1-m}$, and obtain
\[
r(x)=\sum_{i=0}^{n-1-m} \alpha_i P_i(x).
\] 
Let
\[
w(x):=\sum_{i=0}^{n-1-m} \lfloor\alpha_i\rfloor P_i(x)
\]
and
\[
h_n(x):= x^{n}+\sum_{i=n-m}^{n-1} a_ix^i +4w(x)-2.
\]
By the Eisenstein criteria at the prime 2, $w(x)$ is irreducible. Moreover,
\[
h_n(x)=p_{\mathcal Z}Q_m(x)+\sum_{i=0}^{n-1-m} \beta_iP_i(x),
\]
where $|\beta_i|=4|\alpha_i-\lfloor\alpha_i\rfloor|<4.$ Note that
\[
|P_i(z)|\leq \lambda_i e^{(n-m)U_\mu(z)}
\]
for every $z\in \mathbb{C}$ where by Proposition~\ref{upperbound},
\[\lambda_i\leq n^{cn^{1-\frac{\delta}{12}}}.\]
Moreover, by \eqref{potapp}  and Proposition~\ref{chpol},
\[
\left| p_{\mathcal Z}(z)Q_m(z) \right|\geq  e^{(n-m)U_\mu(z)} \kappa^m n^{-Cn^{1-\frac{\delta}{6}}}
\]
for every $z\notin \Sigma(\rho)$, 
where $M=\lfloor n^{1/3}\rfloor.$
By taking $\frac{\delta}{12}>\delta_0>0$ and $m=n^{1-\delta_0}$, it follows that 
\[
\left| p_{\mathcal Z}Q_m(z) \right|> \left|\sum_{i=0}^{n-m-1} \beta_iP_i(z)\right|
\]
for large enough $n$ and every $z\notin \Sigma(\rho).$
By Rouch\'e's theorem, $h_n(z)$ and $p_{\mathcal{Z}}Q_m(z)$ have the same number of roots inside $\Sigma(\rho).$ Since all roots of $p_{\mathcal Z}Q_m(z)$ are inside $\Sigma(\rho),$ all roots of $h_n(z)$ are also inside $\Sigma(\rho).$ By Lemma~\ref{lem1}, we have
\[
\frac{\log|Q_m(z)|}{m}\leq \max(\log|z|,0)+C_{\Sigma},
\]
where $C_{\Sigma}$ is a constant that only depends on $\Sigma.$ Since $m=n^{1-\delta_0},$ we have 
\[
\frac{\log |h_n(z)|}{n} \leq  U_\mu(z)+ O(n^{-\delta_0}).
\]
This completes the case when $\Sigma\subset \mathbb C$. For $\Sigma\subset \mathbb{R}$, our argument is similar. We apply Proposition~\ref{chpolreal} instead of Proposition~\ref{chpol} and  $p_{\mathcal Z}(x)Q_m(x)$ which has $n$ distinct roots which are $n^{-3}$ apart. We use the intermediate value theorem and the fact that we have $n$ sign changes on $\Sigma_{\mathbb R}(\rho)$ to prove all roots are real and inside $\Sigma_{\mathbb{R}}(\rho)$ instead of the  Rouch\'e's theorem.
\end{proof}

\subsection{Proof of Theorem~\ref{general}}
In this section, we show that Theorem~\ref{main1} implies Theorem~\ref{general}. To prove Theorem~\ref{main1}, we assumed that $\mu$ is a H\"older measure, that the support of $\mu$ has empty interior, and that its complement is connected.
Here we show that these conditions are unnecessary for proving Theorem~\ref{general}. 
By Proposition~\ref{holderreduction}, for any $\rho>0,$  there exists a sequence of H\"older probability measure $\mu_n\in\mathcal{B}_{\Sigma(\rho)}$ such that \( \lim_{n\to \infty} \mu_n=\mu\). So by a diagonal argument the proof is reduced to the case where $\mu$ is H\"older.
This is sufficient for the case $\Sigma\subseteq \mathbb R$, so suppose otherwise.
\newline

Recall $\Sigma\subset [a,b]\times[-c,c]$ as in section \ref{bound_Kn} and
$[a,b]\times[-c,c]= \bigcup_{0\leq i,j< 2M}B_{ij}$,
where we define $B_{ij}=[a_i,a_{i+1}]\times[c_{j},c_{j+1}]$ and $a_i=a+\frac{i(b-a)}{2M}$ and $c_j=-c+\frac{j(2c)}{2M}$. Furthermore, $z_{ijk}\in B_{ij}$ for each $1\le k\le n_{ij}=\lfloor(n+1)\mu(B_{ij})\rfloor+\epsilon_{ij}$. By a covering argument, we can choose the parameter $A$ in \eqref{Bireal} to be $-\frac{1}{2}\log_n(n^{-1}M^{-2+\delta})\le A<1$. Here we take $M=\lfloor n^{1/3}\rfloor$ and $A=\frac{5}{6}$.
We now define a measure which imitates $\mu$ by being uniformly distributed on a union of intervals around the $z_{ijk}$.
Let $\tilde \mu$ be the measure supported on $\bigcup_{ijk}\{z:|\Re(z-z_{ijk})|<\frac{1}{2(n+1)},\Im(z)=\Im(z_{ijk})\}$   
where $\tilde \mu$ restricted to each interval is the one-dimensional Lebesgue measure on that interval.
Note that $\tilde\mu$ is a probability measure since $|z_{ijk}-z_{i'j'k'}|\gg n^{-\frac{5}{6}}$.
Now let $\tilde\nu$ be the equilibrium measure on the support of $\tilde\mu$, and let $\tilde \nu_{\epsilon}$ be the normalized pushforward of $\tilde\nu$ under the scaling map $x\mapsto (1+\epsilon)x$.
Finally, define $\mu_{\epsilon,\epsilon'}=(1-\epsilon')\tilde\mu+\epsilon'\tilde\nu_\epsilon$ for $\epsilon,\epsilon'\in(0,1)$.

\begin{proposition}\label{no_interior_bound}
For all $Q\in\mathbb Z[x]$ and $0<\epsilon,\epsilon'<\frac{1}{2}$, $\int \log|Q(x)|d\mu_{\epsilon,\epsilon'}(x)\ge 0$ for large enough $n$.
\end{proposition}
\begin{proof}First, we relate $U_{\tilde{\mu}}$ with $U_\mu$. We compute
$$U_{\tilde\mu}(x)=\int\log|x-z|d\tilde\mu = \sum_{ijk}\int_{\frac{-1}{2(n+1)}}^\frac{1}{2(n+1)}\log|x-y -z_{ijk}|dy.$$ 
Let $\hat{z}$ be the closest root of $P_{\mathcal{Z}}$ to $x$. Then $\forall ijk\neq hrs$, $|x-z_{ijk}|\gg n^{-5/6}$ by \eqref{Bij} choosing $M= \lfloor n^{1/3}\rfloor$.
So for $ijk\neq hrs$, we have that 
$$\int_{\frac{-1}{2(n+1)}}^\frac{1}{2(n+1)}\log|x-y-z_{ijk}|dy - \frac{1}{n}\log|x-z_{ijk}| = \int_{\frac{-1}{2(n+1)}}^\frac{1}{2(n+1)}\log\left|1-\frac{y}{x-z_{ijk}}\right| dy  - \frac{\log|x-z_{ijk}|}{n(n+1)}.$$
Since $\log\left| 1-\frac{y}{x-z_{ijk}}\right|\le \frac{|y|}{|x-z_{ijk}|}$, this is $O(n^{-7/6} + \frac{|\log|x-z_{ijk}||}{n^2})$.
Thus if $|x-\hat{z}|\ge n^{1/6}$, we have by Lemma \ref{lem1} that
$$U_{\tilde\mu}(x)=U_{\mu}(x)+O(n^{-1/6}).$$
Otherwise,
$$U_{\tilde\mu}(x)=\int_{\frac{-1}{2(n+1)}}^\frac{1}{2(n+1)}\log|x-y-\hat{z}|dy
+ \frac{1}{n}\sum_{ijk\neq hrs}\log|x-z_{ijk}|+O(n^{-1/6})$$
In the latter case, $\int_{\frac{-1}{2(n+1)}}^\frac{1}{2(n+1)}\log|x-y-\hat{z}|dy \ll\max(n^{-1}\log(n),\int_{\frac{-1}{2n}}^\frac{1}{2n}\log|y|dy)\ll n^{-1}\log(n)$.
From our proof of Proposition \ref{simplex},
$$
\frac{\log |p_{\mathcal{Z}}^{efg+}(x)|}{n} - U_{\mu}(x)= \frac{\log \frac{|x-\hat{z}||x-\Re(z_{efg})|}{|x-z_{efg}||x-\overline{z_{efg}}|}}{n} +O\left(\frac{M^2\log(n)L}{n}+ M^{-\delta/2}\log(n)^{1/2}\right).
$$
Taking $M\sim n^{1/3}$ and $L\sim n^{\frac{1}{3}-\frac{\delta}{6}}$, we have now shown that
$$U_{\tilde\mu}(x) = U_\mu(x)
+O\left(n^{\frac{-1}{6}}+n^{\frac{-\delta}{6}}\log(n)\right).$$
Therefore in either case, we have $U_{\tilde\mu(x)}=U_\mu(x)+O(n^{-1/6}+n^{-\delta/6}\log(n))$.
\newline

Now for any $Q(x)=c\prod_{i=1}^m(x-\alpha_i)\in\mathbb Z[x]$, we have that
$$\int\log|Q(x)|d\mu_{\epsilon, \epsilon'} =(1-\epsilon')\int\log|Q(x)|d\tilde\mu + \epsilon'\int\log|Q(x)|d\tilde\nu_\epsilon$$
Using our calculation above,
$$\int\log|Q(x)|d\tilde\mu = \log|c| + \sum_{i=1}^mU_{\tilde \mu}(\alpha_i) = \int \log|Q(x)|d\mu 
+ mO(n^{-1/6}+n^\frac{-\delta}{6}\log(n))$$
Since $I_{\tilde\nu}\ge I_\mu\ge0$, we have that
\begin{equation}\label{pushcap}I(\tilde\nu_\epsilon) = \int\int\log|(1+\epsilon)(x-y)|d\tilde\nu d\tilde\nu=I(\tilde\nu)+\log|1+\epsilon|>I(\tilde\nu)+\epsilon-\epsilon^2.\end{equation}
We now bound $I(\tilde\nu)$. Since $\tilde \nu$ is the equilibrium measure on the support of $\tilde \mu$, $I(\tilde \mu)\le I(\tilde\nu)$.
To compute $I(\tilde\mu)$, we first see that
$\int_{\frac{-1}{2(n+1)}}^\frac{1}{2(n+1)}\int_{\frac{-1}{2(n+1)}}^\frac{1}{2(n+1)}\log|x-y|dxdy\ll n^{-2}\log(n)$.
Therefore
$$I(\tilde\mu)=\sum_{ijk\neq i'j'k'}\int_{\frac{-1}{2(n+1)}}^\frac{1}{2(n+1)}\int_\frac{-1}{2(n+1)}^\frac{1}{2(n+1)}\log|z_{ijk}-z_{i'j'k'}+x-y|dxdy+O(n^{-1}\log(n)).$$
By the same argument before, this is
$\frac{1}{n^2}\sum_{ijk\neq i'j'k'}\log|z_{ijk}-z_{i'j'k'}|+O(n^{-1/6})$.
Applying \eqref{maineq} with $M\sim n^{1/3}$ and $L\sim n^{\frac{1}{3}-\frac{\delta}{6}}$,
$I(\tilde\mu) = O(n^{-1/6}+n^{\frac{-\delta}{6}}\log(n))$. Therefore for large enough $n$, $|I(\tilde \mu)|\le \frac{\epsilon}{4}$. Using \eqref{pushcap}, we have $I(\tilde\nu_\epsilon)>\frac{\epsilon}{4}$ since $\epsilon<1/2$.
Thus since
$$\int\log|Q(x)|d\mu_{\epsilon,\epsilon'}=(1-\epsilon')\int\log|Q(x)|d\mu+m\left(\epsilon'I(\tilde\nu_\epsilon)+(1-\epsilon')O(n^{-1/6}+n^\frac{-\delta}{6}\log(n))\right)$$
by equation (I.1.9) of \cite{logpotentials}, this shows that $\int\log|Q(x)|d\mu_{\epsilon,\epsilon'}>0$
for sufficiently large $n$. So for each $0<\epsilon,\epsilon'<\frac{1}{2}$, there is a large enough $n$ for which $\int\log|Q(x)|d\mu_{\epsilon,\epsilon'}>0$ as desired.
\end{proof}

\begin{proof}[Proof of Theorem~\ref{general}]
Suppose that $d_{\Sigma}<1.$ Note that
\[
\log d_{\Sigma}= I(\mu_{eq})=\sup_{\mu\in \mathcal{P}_{\Sigma}}I(\mu),
\]
where  $\mu_{eq}$ is the equilibrium measure of $\Sigma.$
Hence, 
\[
I(\mu)<0
\]
for every $\mu\in \mathcal{P}_{\Sigma}.$ By Theorem~\ref{mdim} for polynomial $Q(x,y)=x-y$, 
\[
\mathcal{B}_{\Sigma}=\emptyset.
\]
Since $\mathcal{A}_{\Sigma}\subset \mathcal{B}_{\Sigma},$ this completes the proof of Theorem~\ref{general} if $d_{\Sigma}<1.$
\\

Otherwise, $d_{\Sigma}\geq 1$ and $\mu_{eq}\in \mathcal{B}_{\Sigma}.$ 
This implies that $\mathcal{B}_{\Sigma}\neq \emptyset.$
Denote by $\tilde\mu_{\epsilon,\epsilon'}$ the measure described in Proposition \ref{no_interior_bound} for the minimal value of $n$ giving the theorem. This gives us that $\tilde\mu_{\frac{1}{m},\frac{1}{m}}\stackrel{\ast}{\rightharpoonup}\mu$ as $m\to\infty$. Further, for each $m$, there is a sequence $P_{m,n}$ of polynomials with $\mu_{P_{m,n}}\stackrel{\ast}{\rightharpoonup}\tilde\mu_{\frac{1}{m},\frac{1}{m}}$ as $n\to \infty$. Taking a subsequence $P_{m,n_m}$, we have $\mu_{P_{m,n_m}}\stackrel{\ast}{\rightharpoonup}\mu$ as $m\to\infty$. 
Therefore, we have shown that the assumptions that the support of $\mu$ must have non-empty interior and have its complement be connected are unnecessary. Theorem~\ref{general} follows from Theorem~\ref{main1}. 
\end{proof}
Finally we give a proof for Corollary~\ref{gsmith}.
\begin{proof}[Proof of Corollary~\ref{gsmith}]
    Suppose that $\Sigma=\bigcup_{i}(R_i \cup \bar{R_i}) \bigcup_{j}I_j$ with $d_{\Sigma}>1$. 
   By \cite[Proposition 2.4]{Smith}, there exists a sequence of probability measures $\mu_n\in\mathcal{B}_\Sigma$ with H\"older exponent at least $1/2$ such that   \( \lim_{n\to \infty} \mu_n=\mu\). So we may assume without loss of generality that $\mu$ is H\"older. By Proposition~\ref{no_interior_bound}, we may further assume that the support of $\mu$ is inside $\Sigma$, has empty interior, and has connected complement.
   \newline

    By continuity of the transfinite diameter~\cite[Theorem J]{MR72941}, it is possible to find a subset $\Sigma_1\subset \Sigma$ which is a finite union of rectangles and intervals such that ${\Sigma_{1\mathbb{R}}}(\rho)\subset \Sigma$ and $d_{\Sigma_1}> 1$. We can see this as follows. Let $\Sigma=\bigsqcup_{i}(R_i \sqcup \bar{R_i}) \bigsqcup_{j}I_j.$ Let $\ell$ be the length of the longest side of the rectangles and intervals, and let $d$ be the smallest distance between any pair of distinct rectangles or intervals if there are multiple components and 1 otherwise.
    Now let $0 < \epsilon < \frac{\min(\ell,d)}{2\ell}\left(1-\frac{1}{d_{\Sigma}}\right)$.
    For a given $x\in\Sigma$, let $c_x$ be the center of the rectangle or interval it lies in. Let $\tilde x=x + \epsilon(c_x-x)$. This essentially shrinks each rectangle/interval towards its center by a factor of $1-\epsilon$.
    Thus for $x,y\in\Sigma$, if $c_x=c_y$, then $|\tilde x - \tilde y|=(1-\epsilon)|x-y|>\frac{1}{d_\Sigma}|x-y|$.
    Otherwise, we have $|\tilde x - \tilde y|\ge|x-y|-2\epsilon \ell$ by the triangle inequality since $|c_x-x|<\ell$.
    Note that if $c_x\neq c_y$, then $|x-y|>d$ in which case
    $$|\tilde x - \tilde y|\ge|x-y|-2\epsilon \ell > \frac{1}{d_\Sigma}|x-y|.$$
    Now let $x_{n,1},x_{n,2},\dots, x_{n,i}$ be Fekete points of $\Sigma$ for each $n\in\mathbb N$. Then
    $$d_{\Sigma_1}\ge \lim_{n\to\infty}\prod_{1\le i<j\le n}|\tilde x_{n,i} - \tilde x_{n,j}|^\frac{1}{\binom{n}{2}} > \frac{1}{d_\Sigma}\lim_{n\to\infty}\prod_{1\le i<j\le n}|x_{n,i} - x_{n,j}|^\frac{1}{\binom{n}{2}}=1.$$
    Thus there is some finite disjoint union of rectangles and intervals $\Sigma_1$ and $\rho>0$ for which $\Sigma_1(\rho)\subset\Sigma$ with $d_{\Sigma_1}>1$ as claimed.\footnote{In a private communication, Professor Rumely kindly outlined a different argument for continuity of the capacity for more general $\Sigma$.}
    \newline
    
    Note that the argument in the proof of Theorem~\ref{main1} in section ~\eqref{rouche} works with a slight modification. More precisely,  we construct $p_{\mathcal{Z}}$ as in the proof of Theorem~\ref{main1} and apply Proposition~\ref{chpolreal} to $\Sigma_1$ and $P_{\mathcal{Z}}$ to obtain $Q_m(x)$ such that 
      $p_{\mathcal Z}(x)Q_m(x)$ has $n$ distinct roots inside $\Sigma$ which are $n^{-3}$ apart from the boundary of $\Sigma$. Moreover, we have
      \[
\left| p_{\mathcal Z}(z)Q_m(z) \right|\geq |z-\hat{z}| e^{(n-m)U_\mu(z)} \kappa^m n^{-Cn^{1-\frac{\delta}{6}}}.
\]
where $|z-\hat{z}|=\min_{p_{\mathcal Z}(\alpha)Q_m(\alpha)=0}|z-\alpha|$. We use the above inequality and the intermediate value theorem on $\bigcup I_j$ to prove that the number of real roots of the two polynomials $p_{\mathcal Z}(x)Q_m(x)$ and $h_n(x)=p_{\mathcal Z}Q_m(x)-\sum_{i=0}^{n-1-m} \beta_iP_i(x)$ are the same. We apply the above inequality and Rouch\'e's theorem on the boundary of $\bigcup_{i}(R_i \cup \bar{R_i})$ to prove the number of complex roots of $p_{\mathcal Z}(x)Q_m(x)$ and $h_n(x)$ are the same inside $\bigcup_{i}(R_i \cup \bar{R_i})$. This completes the proof of our corollary. 
\end{proof}

\section{The Algorithm}\label{main}
\subsection{Lattice Algorithms}\label{lattice_alg}
The algorithm we develop requires access to short vectors in a lattice. As seen in the work by Micciancio ~\cite{SVP}, finding a vector whose length is within a constant multiple of the shortest vector in a lattice is NP-hard to compute. This is known as the shortest vector problem or SVP. It is still desirable  in general to have access to short vectors in a lattice. An algorithm known as the LLL-algorithm, developed by Lenstra, Lenstra, and Lov\'asz, finds a basis of vectors in a lattice in polynomial time ~\cite{LLL} where the shortest vector of the output has length within an exponential factor of the shortest vector of the lattice. In the implementation of this algorithm, we use the built-in method in PARI/GP named \texttt{qflll}. When the \texttt{qflll} method is given a matrix $A$ whose columns generate a lattice, it returns a transformation matrix $T$ so that $AT$ is an ``LLL-reduced" basis of the lattice. While the LLL-algorithm is quite reasonable in practice, the asymptotic bounds are not strong enough to prove a sub-exponential factor. \newline

In 1986, Schnorr~\cite{Schnorr} devised a parametrized family of algorithms attacking the SVP problem which output reduced bases somewhere between the stringent requirements of Korkine-Zolotarev reduction and the looser requirements of LLL reduction based on the input parameter. Careful parameter selection achieves slightly sub-exponential constant factors in polynomial time. Schnorr proved that if the algorithm produces a basis $b_1,\dots, b_n$, then $||b_1||$ is within a subexponential factor of a shortest vector in the lattice. We prove that in addition, if $\lambda_1,\dots, \lambda_n$ are so that $\lambda_i$ is the least value for which there are $i$ linearly independent vectors of norm at most $\lambda_i$, then for each $i=1,\dots, n$, $||b_i||$ is within a sub-exponential factor of $\lambda_i$.
We now introduce some notation from Schnorr's work and  recall some of his results.
\newline

Let $b_1,\dots, b_n$ be a basis. Then let $b_i^*$ be the projection of $b_i$ onto the orthogonal complement of $\text{span}\{b_1,\dots, b_{i-1}\}$ for $i=1,\dots, n$. We can get $b_1=b_1^*,b_2^*,\dots,$ and $b_n^*$ by Gram-Schmidt process without normalizing.
Now define $\mu_{i,j}:=\frac{\langle b_i,b_j^*\rangle}{||b_j^*||^2}$.
\begin{definition}We say $b_1,\dots, b_n$ is size-reduced if $\mu_{i,j}\le 1/2$ for all $i>j$.\end{definition}
\begin{definition}
Let $\Lambda_i$ be the lattice spanned by the projections of $b_i,b_{i+1},\dots, b_n$ onto the orthogonal complement of $\text{span}\{b_1,\dots, b_{i-1}\}$. Then $b_1,\dots, b_n$ is Korkine-Zolotarev reduced if $||b_i^*||$ is minimal among non-zero vectors in $\Lambda_i$.
\end{definition}
\begin{definition}Let $k\mid n$. We say that $b_1,\dots, b_n$ is semi $k$-reduced if it is size-reduced and it satisfies both
\begin{equation}
    ||b_{ik}^*||^2 \le 2||b_{ik+1}^*||^2,
\end{equation}
and the components of $b_{ik+j}$ for $j=1,\dots, k$ orthogonal to $b_1,\dots, b_{ik-1}$ are Korkine-Zolotarev reduced for each $i=0,\dots,\frac{n}{k}-1$. Call this property (KZ).
\end{definition}
Lastly, the constant $\alpha_k$ is defined as $\max\frac{||b_1||^2}{||b_k^*||^2}$ where the max is over all Korkine-Zolotarev reduced bases on lattices of rank $k$. Corollary 2.5 in Schnorr's paper shows that $\alpha_k\le k^{1+\log k}$.
\newline

\begin{theorem}\label{factors}
Let $b_1,\dots,b_n$ be a semi $k$-reduced basis and $\lambda_j$ be the least real number so there are $j$ linearly independent vectors in the lattice spanned by $b_1,\dots, b_n$ of norm at most $\lambda_j$. Then $||b_j||\le k^{\frac{n}{k}\log k +O(\frac{n}{k}+\log k)}\lambda_j$.
\end{theorem}

The essence is a combination of the proofs of Theorems 2.3 and 3.4 in ~\cite{Schnorr}.

\begin{proof}
Fix $1\le j\le n$. Suppose that $v_1,\dots, v_j$ are linearly independent vectors in the lattice with $||v_i||=\lambda_i$ for $i=1,\dots, j$.
For each $i$, write $v_i = \sum_{s=1}^n c_{s,i}b_s$.
Let $s_i$ be the largest index for which $c_{s_i,i}\neq 0$. There is some $i=1,\dots, j$ so that $s_i\ge j$ by linear independence. If $c_{s,i}^*$ is so that $v_i = \sum_{s=1}^nc_{s,i}^*b_s^*$, then $c_{s_i,i}=c_{s_i,i}^*\in \mathbb Z$.
This implies that
$\lambda_j = ||v_j||\ge  ||v_i||\ge||c_{s_i,i}^*b_{s_i}^*||\ge ||b_{s_i}^*||$. Now note that for each $0< s<t\le k$ and $m$ so that $(m+1)k\le n$, then $||b_{mk+s}^*||\le \alpha_k||b_{mk+t}^*||$ as stated in the proof of Theorem 2.3 of ~\cite{Schnorr} and call this inequality $(\ast)$.
\newline

Using the fact that $b_1,\dots, b_n$ is size reduced,
$$||b_j||=\left|\left|b_j^*+\sum_{s=1}^j\mu_{j,s}b_s^*\right|\right|\le \sum_{s=1}^j||b_s^*||.$$

To bound $||b_s^*||$ for $1\le s\le j$, suppose $s=m_1k+t_1$ and $s_i=m_2k+t_2$ for $0\le t_2,t_2<k$. Then
\begin{align*}
    ||b_s^*|| &\le \alpha_k||b_{(m_1+1)k}^*|| &\text{by property (KZ) and }(\ast)\\
    &\le 2\alpha_k ||b_{(m_1+1)k+1}^*||&\text{by }(6)\\
    &\le (2\alpha_k)^{m_2-m_1}||b_{m_2k+1}||&\text{by induction}\\
    &\le (2\alpha_k)^{m_2-m_1}\alpha_k||b_{s_i}^*||&\text{by property (KZ) and }(\ast)\text{ since }j\le s_i\\
    &\le (2\alpha_k)^{\frac{n}{k}-\lfloor\frac{s}{k}\rfloor}\alpha_k\lambda_j&\because ||b_{s_i}^*||\le \lambda_j.
\end{align*}
Since $\alpha_k \le k^{1+\log k}$, we have $||b_s^*||\le k^{(\frac{n}{k}+1)(\log k+2)-\lfloor\frac{s}{k}\rfloor\log k}$.
Therefore,
\begin{align*}
||b_j||&\le k^{\frac{n}{k}\log k + O(\frac{n}{k}+\log k)}\sum_{s=0}^\infty k^{-\lfloor\frac{s}{k}\rfloor \log k}\lambda_j\\
&\le k^{\frac{n}{k}\log k + O(\frac{n}{k}+\log k)}k\sum_{s=0}^\infty k^{-s\log k}\lambda_j&\text{since every group of }k\text{ is equal}\\
&\le k^{\frac{n}{k}\log k +O(\frac{n}{k}+\log k)}\left(1-k^{-\log k}\right)^{-1}\lambda_j\\
&\le k^{\frac{n}{k}\log k +O(\frac{n}{k}+\log k)}\lambda_j&\text{since }(1-k^{-\log k})^{-1}\le 3\text{ for }k\ge 2.
\end{align*}
This completes the proof.
\end{proof}

According to Schnorr's paper, this algorithm takes $O(n^4\log B+k^{\frac{k}{2}+o(k)}n^2\log B)$ arithmetical steps on $O(n\log B)$ integers where $B$ is the Euclidean length of the longest input vector. In our case, $B$ is a constant depending on $\mu$. Taking $k=\frac{\log(n)}{\log(\log(n))}$, this is $O(n^4 + n^2\log(n)^{\frac{3\log n}{2\log\log n}})$ which is contained in $O(n^4)$ and thus this choice of $k$ makes the algorithm polynomial time in $n$.
In this case, the factor in Theorem \ref{factors} is at most $e^{\frac{2n\log(\log (n))^{3}}{\log(n)}}$ for large enough $n$; in particular, it is sub-exponential.

\subsection{Geometry of numbers}
There are two different sub-methods of our algorithm regarding geometry. The first finds a short basis of a given lattice.
The second finds an integral polynomial near a real polynomial where distance is measured with respect to a given basis. In these two algorithms, $P$ is an integral, square-free, monic polynomial with roots $\alpha_1, \dots, \alpha_n$.

\subsubsection{Finding Short Lattice Bases}\label{lattice_bases}
To find short lattice bases, we implement a version of Proposition ~\ref{mink}. This theorem shows the existence of a short basis of polynomials in a given convex space. In particular, let $K$ be the set of real polynomials of degree at most $n$ with $\log|P(x)|\le nU_\mu(x)$.
There are linearly-independent, integral polynomials $P_0,\dots,P_n$ so that for each $0\le k\le n$, if $P_k\in \lambda K$ where $\lambda\in\mathbb{R}$ is minimal, then there is no integral polynomial in $\rho K$ for any $\rho<\lambda$ linearly-independent from $P_0,\dots, P_{k-1}$. We showed that the largest of these $\lambda$'s is sub-exponential in $n$ in Proposition~\ref{upperbound}. Thus $\{P_0,\dots, P_n\}$ forms a basis of lattice points in a subexponential multiple of $K$. By Theorem \ref{factors}, we can compute explicit integral polynomials which are within a sub-exponential factor of $K$ in polynomial time in $n$. This is formalized in Corollary \ref{find_lattice_basis}.
For the rest of this section, suppose $\Sigma\subset\mathbb C$ is compact with empty interior and connected complement.

\begin{corollary}\label{find_lattice_basis}
If $\mu\in\mathcal B_\Sigma$ is a H\"older probability measure, then we can compute a basis $Q_0,\dots, Q_n$ of integral polynomials of degree at most $n$ in polynomial time where each $Q_i\in e^{Cn\frac{\log(\log(n))^3}{\log(n)}}K_n$ for some fixed constant $C$ depending only on $\Sigma$ and $\mu$. More specifically, this runs in $O(n^4)$ time and assuming $n$ is sufficiently large, we can take $C=3$.
\end{corollary}
\begin{proof}
First compute the basis $p_{\mathcal{Z}}^{efg\pm}$ of polynomials of degree less than $n$ as defined in \eqref{pmdef} choosing $M = \lfloor n^{1/3}\rfloor$. Call this basis $\mathcal S$. We want to find an integral basis which is short in the basis $\mathcal S$ by Proposition \ref{simplex}.
Now write $\mathcal E:=\{1,x,\dots, x^{n-1}\}$ in the basis $\mathcal S$ and apply Schnorr's algorithm to get a semi $k$-reduced basis choosing $k=\frac{\log n}{\log\log n}$. We showed at the end of Section \ref{lattice_alg} that this runs in $O(n^4)$ time and that the output polynomials $Q_0,\dots, Q_n$ are so that $Q_i\in e^{2n\cdot\frac{\log(\log(n))^3}{\log n}}n^{cn^{1-\frac{\delta}{12}}}K_n$ by Proposition \ref{upperbound} for each $i$ and fixed $c$ for large enough $n$.
It follows that $Q_i\in e^{3n\cdot\frac{\log(\log(n))^3}{\log n}}K_n$ for large enough $n$.
\end{proof}

One way of computing the integral polynomials is as follows. Suppose $T$ is a matrix so that $AT$ is a reduced form (as in Schnorr's algorithm) of $A$ where $A$ is the matrix of $\mathcal E$ written in base $\mathcal S$.
Then $AT$ has linearly independent columns representing such polynomials in base $\mathcal S$. So $T$ has columns representing those polynomials in base $\mathcal E$ as desired. This gives the desired $n+1$ linearly independent integer polynomials.
Furthermore, note that that doing the change of basis via matrix inversion may not be numerically stable for large matrices. Fortunately, we can compute the change of basis matrix explicitly. As is typically done, the time bounds given in this paper are given in the number of arithmetic operations. Numerical stability is not studied in this paper.

\subsubsection{Finding Close Integer Polynomials}\label{int_poly}
The following discussion allows us to algorithmically produce polynomials like those whose existence is proved in Theorem 3.2 of \cite{Smith}. The intention of this method is to find an integral polynomial close to a given real polynomial in polynomial time where distance is measured with respect to a given basis. In our case, we use the basis produced by Corollary \ref{find_lattice_basis}.

\begin{theorem}\label{close_poly}
Given a H\"older probability measure $\mu\in\mathcal B_\Sigma$, there is a polynomial time algorithm which takes in a real polynomial $Q(z)$ of degree less than $n$ and produces an integer polynomial $H$ of degree less than $n$ such that $\frac{\log(Q-H)(x)}{n}\le U_\mu(x) + Cn\cdot \frac{\log(\log(n))^3}{\log(n)}$. The run-time is $O(n^4)$ where $C$ is dependent only on $\mu$ and $\Sigma$.
Assuming $n$  is sufficiently large, we can choose $C=4$.
\end{theorem}
\begin{proof}
Let $Q_0,\dots, Q_{n-1}$ be the output of our algorithm described in Corollary \ref{find_lattice_basis}. Call this basis $\mathcal Q$. Suppose the $j$-th coordinate of $[Q]_\mathcal Q$ is $c_j$.
Take $\tilde c_j$ to be the nearest integer (rounding up if $c_j\in\frac{1}{2}\mathbb{Z}$), and let $H$ be the polynomial so that $[H]_\mathcal Q=\begin{pmatrix}\tilde c_0\\\vdots\\\tilde c_{n-1}\end{pmatrix}$.
Then $H$ is an integer polynomial since $[H]_\mathcal Q$ is an integer linear combination of the integral polynomials in the basis $\mathcal Q$.
We know $Q-H = \sum_{j=1}^n(c_j-\tilde c_j)Q_j$. Thus $Q-H\le \frac{n}{2}e^{3n\cdot\frac{\log(\log(n))^3}{\log(n)}}e^{nU_\mu}$ and so $Q-H\le e^{\frac{4n\log(\log(n))^3}{\log(n)}}e^{nU_\mu}$ for large enough $n$ as desired.
For an analysis of the run-time, we note that Schnorr's algorithm is the bottleneck. It runs in $O(n^4)$ time by Corollary~\ref{find_lattice_basis}.
\end{proof}

\subsection{Main Algorithm}
Here we prove Theorem \ref{main2}. We have already covered the main concepts. By Corollary \ref{find_lattice_basis}, we can get a basis of integral polynomials in a sub-exponential multiple of $K_n$. If we now apply Theorem \ref{close_poly} to $\frac{1}{4}(p_{\mathcal{Z}} - x^n)+\frac{1}{2}$ to get $H(x)$, we could output $P_n(x):=x^n+4H(x)-2$. This gives us an Eisenstein polynomial in a sub-exponential multiple of $K_n$. If we simply want a polynomial $P$ with small $n$-norm, this would be sufficient; however, we have to introduce some extra steps to ensure $\text{supp}(\mu_{P_n})\subset D$ for a given open $D\supset \Sigma$. The proof will imitate the proof of Theorem \ref{main1}.

\begin{proof}[Proof of Theorem \ref{main2}]
Suppose that $\mu\in\mathcal B_\Sigma$ is a H\"older measure with exponent $\delta$, let $D$ be any open set containing $\Sigma$, and choose $\rho>0$ so that $\Sigma(\rho)\subset D$.
Choose $\mathcal Z=\{x_1,\dots, x_{n-m}\}$ as described in section \ref{bound_Kn} with $L= 1$.
Apply Proposition \ref{chpol} to $p_{\mathcal Z}$ with $m$ chosen later to get the polynomial $Q_m(x)$.
Note that the proof of Proposition \ref{chpol} explicitly computes $Q_m(x)$ in polynomial time.
As in the proof of Theorem \ref{main1}, let $p_{\mathcal Z}(x)Q_m(x)$ be given by $x^n+\sum_{i=0}^{n-1}a_ix^i$.
Apply Theorem \ref{close_poly} to $$h(x):=\frac{1}{4}\left(p_{\mathcal Z}(x)Q_m(x)-x^n-\sum_{i=n-m}^{n-1}a_ix^i\right)+\frac{1}{2}$$
to get an integer polynomial $H(x)$. Finally,
$$p(x):=x^n+\sum_{i=n-m}^{n-1}a_ix^i + 4H(x)-2$$
is an Eisenstein polynomial at $2$.
We see that
\begin{equation}\label{close_res}
|p_{\mathcal Z}(x)Q_m(x)-p(x)|= |4\left(h(x) - H(x)\right) - 2| \le e^{Cn\cdot\frac{\log(\log(n))^3}{\log(n)}}e^{(n-m)U_\mu(x)}\end{equation}
for some constant $C$ depending only on $\mu$ and $\Sigma$ by Theorem \ref{close_poly}. As in the proof of Theorem \ref{main1}, we have
$$|p_{\mathcal Z}(x)Q_m(x)|\ge \kappa^mn^{-cn^{1-\frac{\delta}{6}}}e^{(n-m)U_\mu(x)}$$
for every $x\not\in \Sigma(\rho)$ where $M=\lfloor n^{1/3}\rfloor$.
For large enough $n$, we can choose $m<n$ so that $m=\lceil\frac{2Cn}{\log(\kappa)}\cdot \frac{\log(\log(n))^3}{\log(n)}\rceil$ where we choose $C=4$ as in Theorem \ref{close_poly}.
Then for sufficiently large $n$,
$$|p_{\mathcal Z}(x)Q_m(x)|\ge|p_{\mathcal Z}(x)Q_m(x)-p(x)|$$ for all $x\not\in \Sigma(\rho)$.
By Rouch\'e's theorem, since $p_{\mathcal Z}(x)Q_m(x)$ has all its roots in $\Sigma(\rho)$, so does $p(x)$.
  Furthermore by \eqref{close_res}, $$||p||_n\le e^{Cn\cdot\frac{\log(\log(n))^3}{\log(n)}} + ||p_{\mathcal Z}(x)Q_m(x)||_n\le e^{\tilde{C}n\cdot\frac{\log(\log(n))^3}{\log(n)}}$$
  for some constant $\tilde C$ dependent only on $\Sigma$ and $\mu$.
\end{proof}

\section{Applications}\label{applications}

\subsection{Numerical Data}
We implemented the algorithm described above with some simplifications.\footnote{All examples in section \ref{applications} can be found at https://github.com/Bryce-Orloski/Limiting-Distributions-of-Conjugate-Algebraic-Integers-Applications.} Firstly, we used the LLL-algorithm implemented in PARI/GP named \texttt{qflll}. Secondly, instead of choosing $z_{ijk}$ as indicated in the paper, we use reasonable sampling methods as described in each numerical example given. We also do not apply the extra step of forcing the roots in a $\rho$-neighborhood of the measure's support. So these algorithmic outputs do not employ the full power of the algorithm, but as we see, it still outputs strong results quickly.
All of the examples given in this subsection were computed using Pennsylvania State University's ROAR servers and gave the output in at most two minutes.
\newline

The first example we discuss is pictured in Figure \ref{interval}.
The measure depicted is one in a family which was constructed by Serre~\cite{MR2428512}. In particular, it is the probability distribution $\mu$ on $\Sigma=[a,b]$ with $\mu=c\mu_{[a,b]}+(1-c)\nu_{[a,b]}$ where
 $a=0.1715,$ $ b=5.8255,$  $c=0.5004,$ $\mu_{[a,b]}$ is the equilibrium measure on $[a,b]$, and $\nu_{[a,b]}$ is the pushforward of the equilibrium measure on $[b^{-1},a^{-1}]$ under the map $z\to 1/z$.
 We compute the sample points of $\mu$ using by taking the $k^{th}$ sample point to be the inverse distribution of $\mu$ at $k/n$. The algorithm ran on the ROAR servers for roughly $3$ seconds when asked to compute a degree 100 polynomial (with the sample points pre-computed). The output polynomial has all real roots with their histogram being displayed in Figure \ref{fig2}. The endpoints and interval width in Figure \ref{fig2} are approximate.
 \newline

\begin{figure}[t]
     \centering
     \begin{subfigure}[b]{0.3\textwidth}
         \centering
         \includegraphics[width=\textwidth]{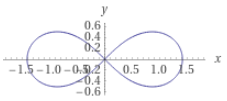}
         \caption{Plot of $|z^2-1|=1$}
         \label{lemniscate_wolfram}
     \end{subfigure}
     \hfill
     \begin{subfigure}[b]{0.3\textwidth}
         \centering         \includegraphics[width=\textwidth]{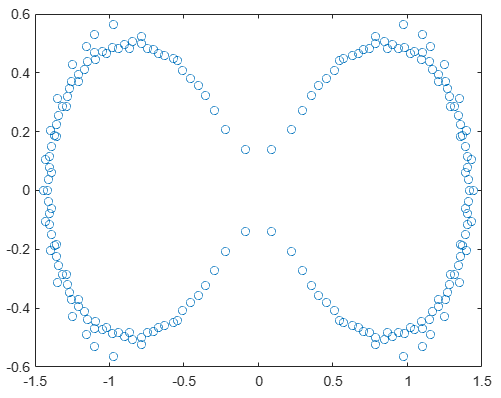}
                \caption{Plotted roots or degree 200 polynomial}
                \label{lemniscate_output}
     \end{subfigure}
     \hfill
     \begin{subfigure}[b]{0.3\textwidth}
         \centering         \includegraphics[width=\textwidth]{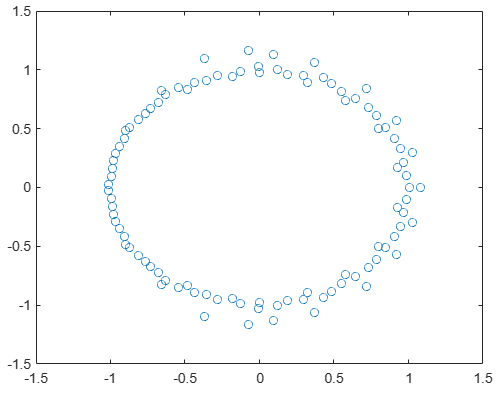}
                \caption{Mapping of output roots via $z\mapsto z^2-1$}
        \label{circle_unif}
     \end{subfigure}
\caption{These demonstrate the effectiveness of the algorithm on a lemniscate.}
\label{lemniscate}
\end{figure}
The second example is pictured in Figure \ref{lemniscate}. Figure \ref{lemniscate_wolfram} depicts the support of the pull-back of the uniform distribution on the unit circle via $|z^2-1|=1$. We sampled this support via this distribution and ran the algorithm on 200 samples. It finished in under two minutes and a plot of its roots are given in Figure \ref{lemniscate_output}.
We want this to sample the pull-back of the uniform distribution on the circle. To see that this is well sampled, Figure \ref{circle_unif} plots the image of these roots under the map $z\mapsto z^2-1$ and we see that it roughly approximates the uniform distribution on the unit circle.
\newline


Our last example is depicted in Figure \ref{annulus}. Here we sampled the annulus $1\le |z|\le 2$ with 100 complex numbers by sampling 50 using the Monte Carlo method and taking their complex conjugates. This is shown in Figure \ref{annulus_pivots}. Using these samples in our algorithm, it returned a degree 100 monic, integral, irreducible polynomial in under two minutes whose roots are plotted in Figure \ref{annulus_output}.
Note that Theorem \ref{main2} only guarantees convergence of the algorithm in polynomial time for $\Sigma$ having empty interior and connected complement. The annulus has neither of these properties though the algorithm still performs well.
\newline
\begin{figure}[t]
     \centering
     \begin{subfigure}[b]{0.4\textwidth}
         \centering
         \includegraphics[width=\textwidth]{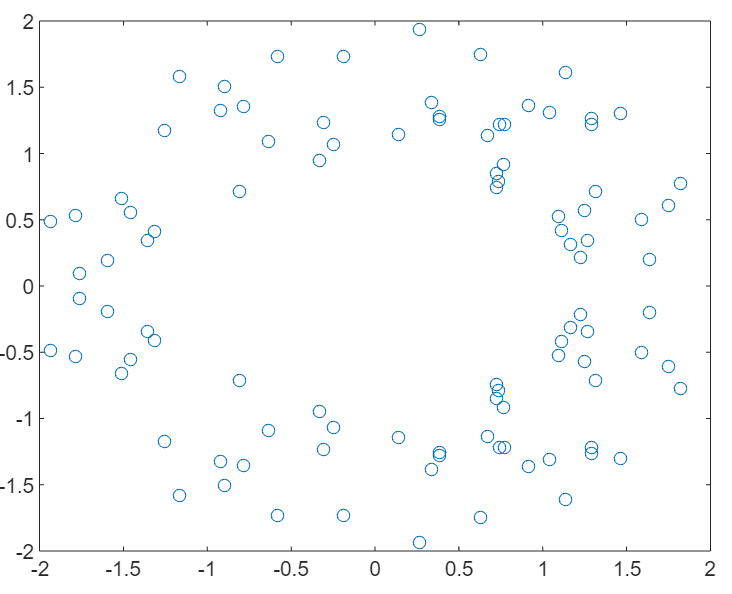}
         \caption{Plot of 50 sampled points from the annulus $1\le |z|\le 2$ and their conjugates}
         \label{annulus_pivots}
     \end{subfigure}
     \hfill
     \begin{subfigure}[b]{0.4\textwidth}
         \centering         \includegraphics[width=\textwidth]{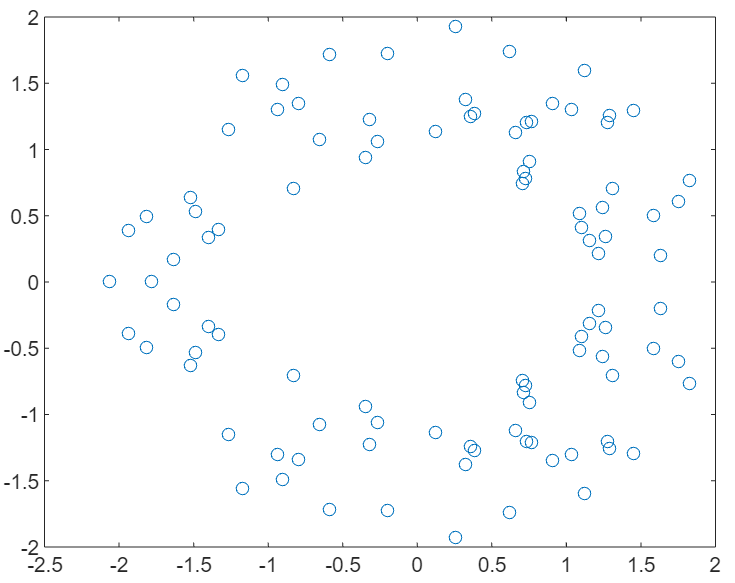}
                \caption{Plotted roots of 100 degree polynomial from algorithm given the samples as input}
                \label{annulus_output}
     \end{subfigure}
\caption{The left shows 50 points sampled uniformly from $1\le |z|\le 2$ and right plots the roots of the output of the algorithm which tried to approximate the distribution of the left plot.}
\label{annulus}
\end{figure}

Notice how we cannot provably show the desired convergence with the LLL-algorithm, and even with Schnorr's algorithm, the convergence we prove is very slow. However, here we are applying the LLL-algorithm and in each of Figures \ref{interval}, \ref{circle_unif}, and \ref{annulus}, we see that the plots of the roots of the polynomial produced by the algorithm are very close to the plots of the corresponding samples.

\subsection{Applications to Abelian varieties}
Recently, Shankar and  Tsimerman~\cite{shankar_tsimerman_2018}, \cite{Tsimerman1} conjectured that for every $g\geq 4$, every Abelian variety defined over $\overline{\mathbb{F}_q}$ is isogenous to the Jacobian of a curve defined over $\overline{\mathbb{F}_q}.$ Given a finite field $\mathbb{F}_q$ and any arbitrary large extension $\mathbb{F}_{q^n}$, we prove that there are infinitely many abelian varieties over $\mathbb{F}_q$ which are not isogenous to the Jacobian of any curve over $\mathbb{F}_{q^n}$. We use Honda Tate theory and construct an arithmetic measure with some conditions on its moments.  First, we introduce our moment conditions which are motivated by the work of  Tsfasman and Vladut~\cite{MR1465522}.
\\

 Let $X$ be a finite curve defined over $\mathbb{F}_q$. Let 
\(
N_r(X):= \# X(\mathbb{F}_{q^r})
\)
be the number of $\mathbb{F}_{q^r}$ of $X.$
By Weil's formula, we have
\begin{equation}\label{positivity}
0
\leq N_r(X)=q^r+1-2q^{r/2}\sum_{j=1}^g\cos(r2\pi \theta_j),    
\end{equation}
where $\sqrt{q}e^{2\pi i \theta_i}$ are Frobenius eigenvalues.
We have 
\[
N_r(X)=\sum_{d|r}dM_d(X),
\]
where $M_d(X)$ is the number of points with degree $d.$ By Mobius inversion formula,
\[
M_r(X)=\frac{1}{r}\sum_{d|r} \mu(d)N_{\frac{r}{d}}(X).
\]
In particular, we have 
\[
\sum_{d|r} \mu(d)N_{\frac{r}{d}}(X) \geq 0
\]
for every $r\geq0.$ 
\\

Let $S_{\sqrt{\alpha}}:=\left\{ \sqrt{\alpha}e^{2\pi i\theta} : \theta \in [0,1) \right\}$ be the circle of radius $\sqrt{\alpha}$ centered at the origin, and define the following probability measure on $S_{\sqrt{\alpha}}$
\[
d\mu_q:=\left(1+\sum_{k=1}^{\infty}\frac{c}{k^2}\cos(k2\pi \theta)\right) d\theta.
\]
where $c=\frac{6}{\pi^210}.$ Note that
\begin{equation}\label{bound}
    0 \leq 0.9 d\theta  \leq d\mu_q \leq 1.1 d\theta.
\end{equation}

\begin{proposition}\label{arithme} Suppose that $\alpha>1.65$ then
 $U_{\mu_\alpha}(z)\geq 0$ for any $z\in \mathbb{C}.$ As a result, $d\mu_\alpha$ is an arithmetic measure. 
\end{proposition}
\begin{proof}
It is enough to show that
    $U_{\mu_\alpha}(z)\geq 0$ for any $z\in \mathbb{C}.$ Suppose that $|z|\geq 2\sqrt{\alpha}$, then 
    \[
     U_{\mu_\alpha}(z)=\int_{S_{\sqrt{\alpha}}} \log|z-\sqrt{\alpha}e^{2\pi i\theta}|d\mu_\alpha \geq \log\sqrt{\alpha}\geq 0.
    \]
    Otherwise, suppose that $|z|<2\sqrt{\alpha}.$
    By \eqref{bound}, we have 
    \[
    \begin{split}
            U_{\mu_\alpha}(z)&=\int_{S_{\sqrt{\alpha}}} \log|z-\sqrt{\alpha}e^{2\pi i\theta}|d\mu_\alpha
            \\
            &=\int_{ |z-\sqrt{\alpha}e^{2\pi i\theta}|<1} \log|z-\sqrt{\alpha}e^{2\pi i\theta}|d\mu_\alpha+ \int_{ |z-\sqrt{\alpha}e^{2\pi i\theta}|>1} \log|z-\sqrt{\alpha}e^{2\pi i\theta}|d\mu_\alpha
            \\
            &\geq 1.1 \int_{ |z-\sqrt{\alpha}e^{2\pi i\theta}|<1} \log|z-\sqrt{\alpha}e^{2\pi i\theta}| d\theta + 0.9\int_{ |z-\sqrt{\alpha}e^{2\pi i\theta}|>1} \log|z-\sqrt{\alpha}e^{2\pi i\theta}| d\theta
            \\
            &=1.1\int_{S_{\sqrt{\alpha}}} \log|z-\sqrt{\alpha}e^{2\pi i\theta}|d\theta- 0.2 \int_{ |z-\sqrt{\alpha}e^{2\pi i\theta}|>1} \log|z-\sqrt{\alpha}e^{2\pi i\theta}| d\theta
            \\
            &\geq  1.1\log(\sqrt{\alpha})-0.2\log(3\sqrt{\alpha})>0. 
    \end{split}
    \]
This completes the proof of our propositions. 
\end{proof}

\begin{corollary}\label{frob_equidistribution}
        There exists a sequence $\{A_g\}$ of ordinary abelian varities  over $\mathbb{F}_q$ such that $\dim(A_g)=g$ and their Frobenious eigenvalues equidistributes with $\mu_q$ as $g\to \infty$.
\end{corollary}
\begin{proof}
Let $\alpha_n:=\sqrt{q}-\frac{1}{10n},$  $\Sigma_n:=[-2\alpha_n,2\alpha_n]$ and 
\(
h_n(z):=z+\frac{\alpha_n}{z}.
\)
Note that $h_n$ is a conformal map with fixed point at infinity and derivative 1 at infinity that sends $\mathbb{C}\backslash S_{\alpha_n}$ to the $\mathbb{C}\backslash \Sigma_n$.
Let $h_nd\mu_{\alpha_n}$ be the push-forward of $\mu_{\alpha_n}$ by $h_n.$ By conformal in-variance of the potential function and Proposition~\ref{arithme}, $h_n d\mu_{\alpha_n}$ has a positive potential function and hence $h_nd\mu_{\alpha_n}$ is arithmetic.
    It follows from  Theorem~\ref{general} that there exists a sequence of irreducible polynomials $\{p_m\}$ with real roots contained in $\Sigma_n(\frac{1}{100 n})\subset [-2\sqrt{q},2\sqrt{q}]$, and $\gcd(p_m(0),q)=1$ with root distribution converging to $h_nd\mu_{\alpha_n}.$ Let 
    \[
    f_m(x):=x^{\deg{p_m}} p_m(x+\frac{q}{x}).
    \]
    Note that $f_m(x)$ has all its roots on $S_{\sqrt{q}}$, and $\gcd(f_m(0),q)=1$. It follows form a Theorem of Serre ~\cite[Theorem~3.1.2]{serre_curves} in the ordinary case that $f_m$ is the charactristic polynomial of some ordinary isogeny class of abelian variety of dimension $m$.
    Now by a diagonal argument and letting $n\to \infty$ and taking $m$ large enough it follows that there exists a sequence of $\{A_g\}$ of ordinary abelian varities  over $\mathbb{F}_q$ such that $\dim(A_g)= g$ and their Frobenious eigenvalues equidistributes  with $\mu_q$ as $g\to \infty$.
\end{proof}

\begin{theorem}\label{eventually_not_jacobian}
    Let $\{A_g\}$ be any family of abelian varieties over $\mathbb{F}_q$ such that $\dim(A_g)=g$ and their Frobenius eigenvalues distribution converges to $\mu_q.$ Given any integer $r\geq 0$, there exists $N$ such that if $g\geq N$ then $A_g$ is not isogenous to the Jacobian of any curve over $\mathbb{F}_{q^r}$.
\end{theorem}
\begin{proof}
    Suppose to the contrary that there exists a sub-sequence $\{ A_{g_i}\}$ of abelian varieties over $\mathbb{F}_q$ such that their Frobenius eigenvalues equidistribute with  $\mu_q$ and also they are isogenous to Jacobian of curves $\{ X_{g_i}\}$ over $\mathbb{F}_{q^r}$ for some $r\geq 0.$ By~\eqref{positivity}, it follows that
    \[
    0 \leq \frac{N_r(X_{g_i})}{g_i}= \frac{q^r+1-2q^{r/2}\sum_{j=1}^{g_i}\cos(r2\pi \theta_j)}{g_i}.
    \]
    By taking the limit of the above as $g_i\to \infty$, we obtain
    \[
    0 \leq -2q^{r/2} \lim_{g_i\to \infty}\frac{\sum_{j=1}^{g_i}\cos(r2\pi \theta_j)}{g_i}=-2q^{r/2} \int \cos(r2\pi \theta) d\mu_q= -q^{r/2}\frac{c}{r^2} <0,
    \]
    which is a contradiction. This completes the proof of our theorem.
\end{proof}

As was mentioned in the introduction, we constructed a polynomial $p(x)$ of degree 290 using the algorithm described in section \ref{main} of the paper. Like the other results in section \ref{applications}, this was implemented with the LLL-algorithm and without using the construction that forces the roots to lie inside a desired support.
The highest order terms were
$$p(x)=x^{290} - 28x^{289} - {484}x^{288} + 20784x^{287} + \cdots.$$
The largest coefficient of this polynomial has 105 digits. The roots all lie inside $[-2\sqrt{3},2\sqrt{3}]\subset\mathbb R$ and so by Honda-Tate theory, it corresponds with an abelian variety over $\mathbb F_3$ with characteristic polynomial $x^{290r}p(x+3/x)^r$ for some exponent $r.$ We checked that the corresponding abelian variety is  ordinary by computing the constant term of $p(x)$ which is coprime to 3. 
Indeed, it is 
$$19875183005444505756032705281200952405059895953380115355156299268417070$$
which is not divisible by 3. It follows from a theorem of Serre~\cite[Theorem~3.1.2]{serre_curves} that in the ordinary case $r=1$; see Example a) on page 24 of~\cite{serre_curves}.\footnote{The authors wish to thank Professor Serre for a fruitful discussion that brought this to the authors' attention as well as assisting the authors in formulating a correct version of Corollary \ref{frob_equidistribution} and Theorem \ref{eventually_not_jacobian}.}
Computing $N_2$ of this variety, we get $-2$. Thus it is not the Jacobian of any curve over $\mathbb F_9$.

\bibliographystyle{alpha}
\bibliography{main}

\begin{thebibliography}{{Smi}21}

\bibitem[Ahl10]{MR2730573}
Lars~V. Ahlfors.
\newblock {\em Conformal invariants}.
\newblock AMS Chelsea Publishing, Providence, RI, 2010.
\newblock Topics in geometric function theory, Reprint of the 1973 original,
  With a foreword by Peter Duren, F. W. Gehring and Brad Osgood.

\bibitem[AP08]{MR2428512}
Juli\'{a}n Aguirre and Juan~Carlos Peral.
\newblock The trace problem for totally positive algebraic integers.
\newblock In {\em Number theory and polynomials}, volume 352 of {\em London
  Math. Soc. Lecture Note Ser.}, pages 1--19. Cambridge Univ. Press, Cambridge,
  2008.
\newblock With an appendix by Jean-Pierre Serre.

\bibitem[BE95]{MR1367960}
Peter Borwein and Tam\'{a}s Erd\'{e}lyi.
\newblock {\em Polynomials and polynomial inequalities}, volume 161 of {\em
  Graduate Texts in Mathematics}.
\newblock Springer-Verlag, New York, 1995.

\bibitem[BR10]{MR2599526}
Matthew Baker and Robert Rumely.
\newblock {\em Potential theory and dynamics on the {B}erkovich projective
  line}, volume 159 of {\em Mathematical Surveys and Monographs}.
\newblock American Mathematical Society, Providence, RI, 2010.

\bibitem[CSZ17a]{Barry}
Jacob~S. Christiansen, Barry Simon, and Maxim Zinchenko.
\newblock Asymptotics of {C}hebyshev polynomials, {I}: subsets of {$\Bbb R$}.
\newblock {\em Invent. Math.}, 208(1):217--245, 2017.

\bibitem[CSZ17b]{chebyshev_R}
Jacob~S. Christiansen, Barry Simon, and Maxim Zinchenko.
\newblock Asymptotics of {C}hebyshev polynomials, {I}: subsets of {$\Bbb R$}.
\newblock {\em Invent. Math.}, 208(1):217--245, 2017.

\bibitem[Fek23]{Fekete1923}
M.~Fekete.
\newblock {\"U}ber die verteilung der wurzeln bei gewissen algebraischen
  gleichungen mit ganzzahligen koeffizienten.
\newblock {\em Mathematische Zeitschrift}, 17(1):228--249, Dec 1923.

\bibitem[FS55]{MR72941}
M.~Fekete and G.~Szeg\"{o}.
\newblock On algebraic equations with integral coefficients whose roots belong
  to a given point set.
\newblock {\em Math. Z.}, 63:158--172, 1955.

\bibitem[Hil94]{MR1554854}
David Hilbert.
\newblock Ein {B}eitrag zur {T}heorie des {L}egendre'schen {P}olynoms.
\newblock {\em Acta Math.}, 18(1):155--159, 1894.

\bibitem[KS21]{Sarnak}
Alicia~J. Koll\'{a}r and Peter Sarnak.
\newblock Gap sets for the spectra of cubic graphs.
\newblock {\em Comm. Amer. Math. Soc.}, 1:1--38, 2021.

\bibitem[LLL82]{LLL}
A.~K. Lenstra, H.~W. Lenstra, Jr., and L.~Lov\'{a}sz.
\newblock Factoring polynomials with rational coefficients.
\newblock {\em Math. Ann.}, 261(4):515--534, 1982.

\bibitem[Mic01]{SVP}
Daniele Micciancio.
\newblock The shortest vector in a lattice is hard to approximate to within
  some constant.
\newblock {\em SIAM J. Comput.}, 30(6):2008--2035, 2001.

\bibitem[MS21]{McKee}
James McKee and Chris Smyth.
\newblock {\em Around the unit circle---{M}ahler measure, integer matrices and
  roots of unity}.
\newblock Universitext. Springer, Cham, [2021] \copyright 2021.

\bibitem[Rob64a]{Robinson}
Raphael~M. Robinson.
\newblock Conjugate algebraic integers in real point sets.
\newblock {\em Math. Z.}, 84:415--427, 1964.

\bibitem[Rob64b]{MR175881}
Raphael~M. Robinson.
\newblock Intervals containing infinitely many sets of conjugate algebraic
  units.
\newblock {\em Ann. of Math. (2)}, 80:411--428, 1964.

\bibitem[Rum13]{MR3154724}
Robert Rumely.
\newblock {\em Capacity theory with local rationality}, volume 193 of {\em
  Mathematical Surveys and Monographs}.
\newblock American Mathematical Society, Providence, RI, 2013.
\newblock The strong Fekete-Szeg\"{o} theorem on curves.

\bibitem[Sch18]{Schur}
I.~Schur.
\newblock \"{U}ber die {V}erteilung der {W}urzeln bei gewissen algebraischen
  {G}leichungen mit ganzzahligen {K}oeffizienten.
\newblock {\em Math. Z.}, 1(4):377--402, 1918.

\bibitem[Sch87]{Schnorr}
C.-P. Schnorr.
\newblock A hierarchy of polynomial time lattice basis reduction algorithms.
\newblock {\em Theoret. Comput. Sci.}, 53(2-3):201--224, 1987.

\bibitem[Ser19]{MR4093205}
Jean-Pierre Serre.
\newblock Distribution asymptotique des {V}aleurs {P}ropres des
  {E}ndomorphismes de {F}robenius [d'apr\`es {A}bel, {C}hebyshev,
  {R}obinson,{$\ldots$}].
\newblock {\em Ast\'{e}risque}, (414, S\'{e}minaire Bourbaki. Vol. 2017/2018.
  Expos\'{e}s 1136--1150):Exp. No. 1146, 379--426, 2019.

\bibitem[Ser20]{serre_curves}
Jean-Pierre Serre.
\newblock {\em Rational points on curves over finite fields}, volume~18 of {\em
  Documents Math\'{e}matiques (Paris) [Mathematical Documents (Paris)]}.
\newblock Soci\'{e}t\'{e} Math\'{e}matique de France, Paris, [2020] \copyright
  2020.
\newblock With contributions by Everett Howe, Joseph Oesterl\'{e} and
  Christophe Ritzenthaler, Edited by Alp Bassa, Elisa Lorenzo Garc\'{\i}a,
  Christophe Ritzenthaler and Ren\'{e} Schoof.

\bibitem[Sie45]{MR12092}
Carl~Ludwig Siegel.
\newblock The trace of totally positive and real algebraic integers.
\newblock {\em Ann. of Math. (2)}, 46:302--312, 1945.

\bibitem[{Smi}21]{Smith}
Alexander {Smith}.
\newblock {Algebraic integers with conjugates in a prescribed distribution}.
\newblock {\em arXiv e-prints}, page arXiv:2111.12660, November 2021.

\bibitem[Smy84]{MR736460}
C.~J. Smyth.
\newblock The mean values of totally real algebraic integers.
\newblock {\em Math. Comp.}, 42(166):663--681, 1984.

\bibitem[ST97]{logpotentials}
Edward~B. Saff and Vilmos Totik.
\newblock {\em Logarithmic potentials with external fields}, volume 316 of {\em
  Grundlehren der mathematischen Wissenschaften [Fundamental Principles of
  Mathematical Sciences]}.
\newblock Springer-Verlag, Berlin, 1997.
\newblock Appendix B by Thomas Bloom.

\bibitem[ST18]{shankar_tsimerman_2018}
ANANTH~N. SHANKAR and JACOB TSIMERMAN.
\newblock Unlikely intersections in finite characteristic.
\newblock {\em Forum of Mathematics, Sigma}, 6:e13, 2018.

\bibitem[ST21]{Tsimerman1}
Ananth~N. {Shankar} and Jacob {Tsimerman}.
\newblock {Abelian varieties not isogenous to Jacobians over global fields}.
\newblock {\em arXiv e-prints}, page arXiv:2105.02998, May 2021.

\bibitem[{Tal}23]{2023arXiv230410021T}
Naser {Talebizadeh Sardari}.
\newblock {Conjugate algebraic integers and 2D-Coulomb's law}.
\newblock {\em arXiv e-prints}, page arXiv:2304.10021, April 2023.

\bibitem[TV97]{MR1465522}
M.~A. Tsfasman and S.~G. Vl\u{a}du\c{t}.
\newblock Asymptotic properties of zeta-functions.
\newblock volume~84, pages 1445--1467. 1997.
\newblock Algebraic geometry, 7.

\end{thebibliography}
\end{document}